\documentclass{article}
\usepackage[T1]{fontenc}
\usepackage[english]{babel}
\usepackage[margin=1.1in]{geometry}
\usepackage[utf8]{inputenc}
\usepackage{lmodern}
\usepackage{microtype}

\usepackage{amsmath,amssymb,amsfonts,amsthm}
\usepackage{mathtools}
\usepackage{mathrsfs} 
\usepackage[square,comma,numbers]{natbib}
\usepackage[colorlinks=true, allcolors=blue]{hyperref}
\usepackage{graphicx}
\usepackage{enumerate}
\usepackage[font=small]{caption}
\usepackage{subcaption}
\usepackage{algorithm}
\usepackage[noend]{algpseudocode}
\usepackage[dvipsnames]{xcolor}
\usepackage{tikz}

\numberwithin{equation}{section}
\def\cC{{\mathcal C}}

\def\cI{{\mathcal I}}
\def\cJ{{\mathcal J}}
\def\cK{{\mathcal K}}

\def\cS{{\mathcal S}}
\def\cU{{\mathcal U}}

\def\cT{{\mathcal T}}

\def\cX{{\mathcal X}}

\def\E{\mathbb{E}}

\def\P{\mathbb{P}}
\def\R{\mathbb{R}}

\def\sA{\mathscr{A}}
\def\sL{\mathscr{L}}

\def\sX{\mathscr{X}}

\renewcommand{\d}{\mathrm{d}}
\newtheorem*{Def*}{Definition}
\newtheorem*{Thm*}{Theorem}
\newtheorem*{Cor*}{Corollary}
\newtheorem*{Rmk*}{Remark}
\newtheorem*{Lem*}{Lemma}
\newtheorem*{Prop*}{Proposition}
\newtheorem*{Asm*}{Assumption}

\newtheorem{Def}{Definition}[section]
\newtheorem{Thm}[Def]{Theorem}
\newtheorem{Cor}[Def]{Corollary}
\newtheorem{Rmk}[Def]{Remark}
\newtheorem{Lem}[Def]{Lemma}
\newtheorem{Prop}[Def]{Proposition}
\newtheorem{Asm}[Def]{Assumption}

\def\we{\rho}  % weight in the Picard iteration 
\def\cld{c_{\delta,\we}}

\def\gl{\gamma_\we}
\def\hcr{\hat{c}_\we}
\def\csr{c_{*,\we}}

\usepackage{fancyhdr}
\title{Iterative Schemes for Markov Perfect Equilibria\footnote{Research of F. H\"ofer, M. Soner and Q. Yan is partially supported 
by the National Science Foundation grant DMS 2406762.}}
\author
{Felix H\"ofer\footnote{Department of Operations Research and Financial
Engineering, Princeton University, Princeton, NJ, 08540, USA, email: 
{\tt fhoefer@princeton.edu}}
\and Mathieu Lauri\`{e}re\footnote{Shanghai Frontiers Science Center of Artificial Intelligence and Deep Learning, NYU-ECNU Institute of Mathematical Sciences, NYU Shanghai, Shanghai, 200126, People’s Republic of China, email: 
{\tt mathieu.lauriere@nyu.edu}}
\and H. Mete Soner\footnote{ 
Department of Operations Research and Financial
Engineering, Princeton University, Princeton, NJ, 08540, USA, email: 
{\tt soner@princeton.edu}}
\and Qinxin Yan\footnote{Program in Applied and Computational
Mathematics, Princeton University, Princeton, NJ, 08540, USA, email: 
{\tt qy3953@princeton.edu}. }}

\date{\today}

\begin{document}
\maketitle

\begin{abstract} 
\noindent 
We study Markov perfect equilibria in continuous-time dynamic games with finitely many symmetric players. The corresponding Nash system reduces to the Nash-Lasry-Lions equation for the common value function, also known as the master equation in the mean-field setting. In the finite-state space problems we consider, this equation becomes a nonlinear ordinary differential equation admitting a unique classical solution. Leveraging this uniqueness, we prove the convergence of both Picard and weighted Picard iterations, yielding efficient computational methods. Numerical experiments confirm the effectiveness of algorithms based on this approach.
\end{abstract}
\vspace{3pt}

\vspace{3pt}


\section{Introduction}
\label{sec:intro}

We study finite-player games under the modeling assumption of symmetry between players, which allows the reduction of the classical Nash system for the value functions to what we refer to as the \emph{Nash-Lasry-Lions (NLL) equation}, inspired by the literature on mean-field games (MFGs). MFGs, independently initiated by Caines, Malhamé and Huang~\cite{HMC1,HMC2,HMC3,HMC4}, and Lasry and Lions~\cite{LL1,LL2,LL3}, approximate such symmetric finite-player games in the large-population limit. Their introduction has spurred extensive research by making otherwise intractable models more analytically and numerically tractable. For historical context and a probabilistic formulation, we refer to the books by Carmona and Delarue~\cite{Car,carmona2018probabilistic}.

What we call the NLL equation corresponds, in the mean-field limit, to the \emph{master equation}, a central object extensively studied in the MFG literature; see, e.g.,~\cite{BA,bensoussan2015master,CDLL,chassagneux2022probabilistic,gangbo2022global}. A structural condition introduced in~\cite{LL3}, known as \emph{Lasry-Lions monotonicity}, is typically used in the mean-field setting to ensure well-posedness of this equation and uniqueness of equilibria. For finite-player games on a finite state space $\cX = \{1, \ldots, d\}$, the NLL equation can be viewed as a system of nonlinear ordinary differential equations derived in~\cite{GMS1,HSY}; see \eqref{eq:N-NLL} below. As shown in the classical work of Gomes, Mohr and Souza~\cite{GMS1}, this system admits a unique solution by standard ODE theory, leading to uniqueness of \emph{Markov perfect equilibria} (see Definition~\ref{def:MPE}) without requiring Lasry-Lions monotonicity.

 In these games, players move randomly on the state space $\cX$ by observing the states of all other players and taking into account their controls and states when making decisions. Their dynamics are described by controlled Markov chains, and they use feedback policies to control their own transition rates. The assumed symmetry implies that it is sufficient to track a single player, referred to as the \emph{tagged} player, along with the empirical distribution of the others.
The tagged player optimizes over their feedback control $\alpha$, which depends on time, their position, and the empirical distribution of the other players. 
To define the equilibrium concept, we assume that all untagged players use a common feedback control $\beta$, which takes the same inputs as $\alpha$. This setup defines the best response $\alpha^*(\beta)$ of the tagged player to a given control $\beta$ of the untagged players. A \emph{Markov perfect equilibrium} (MPE) is a fixed point of this mapping.
We refer the reader to the classical books by Ba\c{s}ar and Olsder~\cite{BO}, and Fudenberg and Tirole~\cite{fudenberg1991game}, for background on dynamic games. For differential games more specifically, see, for example,~\cite{bacsar1986tutorial,Car,Is:99,yong2014differential} and the references therein.

Iterative methods play a central role in the analysis and computation of equilibria, as they offer a concrete procedure for equilibrium formation. The classical fictitious play algorithm, introduced by Brown~\cite{brown} and later generalized in~\cite{LC}, updates each player's strategy by computing a best response to the empirical average of past actions. This idea has been extended in various directions, including to stochastic games~\cite{SPO} and MFGs in both continuous time~\cite{CH} and discrete time~\cite{elie}, where policies are updated at each iteration. Connections between fictitious play and Q-learning~\cite{SPO}, as well as to reinforcement learning formulations of stochastic control~\cite{HWZ,WZ,WZZ}, further highlight the relevance of these methods. In MFGs, convergence of fictitious play has been established under monotonicity or potential structure assumptions~\cite{CB,CH,HS}, and extensive numerical studies, including with common noise~\cite{DV,perrin2020fictitious}, have examined their stability and effectiveness. While fictitious play assigns equal weight to all previous iterates, Picard-type iterations, where players best respond to the most recent strategy, apply exponentially decaying weights to the past. This distinction often leads to improved numerical stability and faster convergence; however Picard-type iterations are not expected to converge in general, particularly when the best response map is not a strict contraction. 
In the context of MFGs, \cite{CCD} also employs Picard iterations to solve the backward stochastic differential equations characterizing the mean-field equilibrium. They prove the convergence of the stochastic processes either under the Lasry-Lions monotonicity conditions or in a short time horizon.
Under these assumptions, also employed in the seminal monograph~\cite{CDLL}, even the open-loop equilibria are unique and the solutions of the NLL equations
are smooth. To study similar backward stochastic differential equations, \cite{HH,Hu} utilize deep neural networks and simulations. For broader overviews of computational methods in MFGs, we refer e.g. to the surveys~\cite{AM,lauriere2021numerical}; for machine learning approaches to games, see e.g.~\cite{hu2024recent}.

In this work, we focus on Picard-type iterations for computing MPEs in symmetric finite-player games. While in general such iterations are not guaranteed to converge, we show, both theoretically and numerically, that they do converge for the class of games considered here, without requiring monotonicity assumptions. Specifically, in the basic Picard iteration, one starts from an arbitrary control $\beta$ and updates it via the best response mapping $\alpha^*(\beta)$. Weighted variants, such as taking convex combinations of the previous $\beta$ and $\alpha^*(\beta)$, can further improve stability. These iterative schemes are closely related in spirit to fictitious play but are tailored to our setting, where each player's policy depends on time, state, and the empirical distribution of others.

In Section~\ref{sec:iterations}, we define the Picard and weighted Picard iterations and prove that they converge to the unique equilibrium under standard convexity conditions, without assuming monotonicity in the measure variable. The proof leverages the uniqueness of the NLL equation and properties of the dynamic programming equation associated with the tagged player's control problem, given a fixed control $\beta$ used by the others. The convergence is geometric, as shown in Theorem~\ref{thm:convergence} and Proposition~\ref{prop:rate}, and provides theoretical guarantees for the numerical methods introduced in Section~\ref{sec:numerical}.  In addition, Theorem~\ref{th:cv-error} shows that numerical errors introduced at each step remain controlled over time, and Proposition~\ref{pro:eN} confirms that the computed strategies remain approximate equilibria even in the presence of such errors. The methods we propose  rely on iterative schemes and avoid directly solving the full NLL equation, which becomes an intractable task for large $N$ due to the number of coupled nonlinear ODEs equal to $d (N+d-1)!/(N!(d-1)!)$. The first approach numerically solves the dynamic programming equation for each $\beta$ using ODE solvers and applies weighted updates; this proves substantially more stable in practice. The second method combines Monte Carlo simulation with empirical risk minimization, which is a technique that has been used to solve optimal control problems in several recent works, see e.g.~\cite{HE,HEJ,hua2025efficient,RS,RST1,RST2,STY}. These algorithms are evaluated in Section~\ref{sec:examples}, where we apply them to the two-state Kuramoto synchronization models from~\cite{cecchin2019convergence, hofer2025synchronization} and a cyber-security model from~\cite{Kolokoltsov_2016}. The results demonstrate fast convergence and numerical stability.

The main contributions of this work are as follows. In short, we show that Markov perfect equilibria can be computed via Picard-type iterations in a class of symmetric, continuous-time, finite-state games. More precisely:
\begin{itemize}
    \item We prove in Theorem \ref{thm:convergence} the convergence of both standard and weighted Picard iterations without requiring monotonicity assumptions with respect to the other players' distribution, and obtain convergence rates in Proposition \ref{prop:rate}.
    \item We establish in Theorem \ref{th:cv-error} the robustness of the scheme to numerical errors in best response computations  by analyzing the propagation of errors through iterations, and in Proposition \ref{pro:eN} show that the scheme computes an approximate equilibrium.
    \item We propose numerical algorithms in Section \ref{sec:numerical} based on dynamic programming or optimization of neural networks, and validate their effectiveness through examples.
\end{itemize}
 
The remainder of the paper is organized as follows. Section~\ref{sec:game} introduces the finite-player game and the associated equations. Section~\ref{sec:iterations} defines the Picard-type iterations, and Section~\ref{sec:proofs} contains the convergence proofs. Section~\ref{sec:errors} analyzes error propagation under numerical approximation. Section~\ref{sec:numerical} presents the proposed computational methods, and Section~\ref{sec:examples} reports on numerical experiments.
\vspace{5pt}

\textbf{Notation. } Let $N$ and $d$ be two positive integers. All constants that we introduce in the sequel may depend on $N$ and $d$, and we do not explicitly show
these dependencies in our notations. We consider continuous-time Markov chains 
in a finite state space $\cX:=\{1,\ldots,d\}$.  There are $N+1$
identical players.  The \emph{state space} is defined as
$\sX:=\cX\times\Sigma^{d-1}_N$,  where
$$
\Sigma^{d-1}_N := \big\{\mu=(\mu_1,\ldots, \mu_d)\ :\
\mu_x \in\{0,1/N,2/N,\ldots,1\}^d \  \text{for  each}\ x \in \cX,\
\text{and}\ \ \mu_1+\ldots+\mu_d=1\big\},
$$
is the  \emph{discretized simplex}.
For $\mu=(\mu_1,\ldots,\mu_d) \in \Sigma^{d-1}_N$ and 
$x \in \cX$, we set $n^\mu_x:= N \mu_x$, so that 
$\mu=\tfrac{1}{N} (n^\mu_1,\ldots, n^\mu_d)$.  
We emphasize that, as $\sX$ is a discrete set, no regularity or
growth assumptions on functions of $(x,\mu)$ are needed in the sequel.

For $x \in \cX$, we will denote by $e_x\in\{0,1\}^d$ the $x$-th unit vector, and define  
$$
	e_{x,z} := \frac1N(e_x-e_z),  \qquad x, z \in \cX.
$$

We set $\R_+:= [0,\infty)$.
For a \emph{rate function} 
$\alpha : \cX \mapsto \R_+^{d-1}$,
and $x, x' \in \cX$ with $x\neq x'$, we write
$\alpha_{x'}(x)$ for the rate of 
transitioning from state $x$ to $x'$.  Conversely, for 
given $x \in \cX$ and $a=(a_1,\ldots,a_{d-1}) \in \R_+^{d-1}$, 
we  define 
the corresponding rate function by $\alpha_{y}(x)= a_{y}$,
for $y<x$,  and $\alpha_{y}(x)= a_{y-1}$, for $y>x$. 

For $k\geq1$ and $z \in \R^k$, $|z|_2^2=(z_1^2+\ldots+z_k^2)$ is the standard Euclidean norm.
For $\varphi :\sX \mapsto \R^k$ we write $|\varphi|^2_2:=\sum_{(x,\mu)\in\sX}|\varphi(x,\mu)|_2^2$.
For $\psi:[0,T] \times \sX \mapsto \R^k$,  with a slight abuse of notation, we define
$\psi(t) : \sX\mapsto\R^k$ by
$\psi(t)(x,\mu):= \psi(t,x,\mu) \in \ \R^k$. Then, 
\begin{equation}
    \label{eq:norm}
    |\psi(t)|_2^2 = \sum_{(x,\mu)\in\sX}|\psi(t,x,\mu)|_2^2,
    \qquad t \in [0,T].
\end{equation}
$\cS^k$ is the Banach space of all continuous functions $\psi :[0,T] \times \sX \mapsto \R^k$ 
endowed  with the norm 
\begin{equation}
\label{eq:snorm}
\|\psi\|_{\cS^k}:=\max_{t\in[0,T]}  \ |\psi(t)|_2.
\end{equation}
We omit the dependence on dimension $k$ when it is clear from the context.

%Finally, we follow the convention that the sum over the empty set is zero.

\section{Finite-player game}
\label{sec:game}

In the section, we introduce a stochastic differential game
between $N+1$ players, indexed by $0,\ldots,N$, in continuous 
time over a finite horizon $[0,T]$ with a finite state space 
$\cX:=\{1,\ldots,d\}$. It is shown in~\cite{HSY}
that when the players are exchangeable,
then all MPEs are symmetric across players and
one can simply tag and follow one player.  Accordingly, 
among the $N+1$ players, 
player $0$ is tagged and called the \emph{tagged player}.
We refer to the remaining untagged $N$ players  as \emph{other players}. 

\subsection{Dynamics}
\label{sec:dynamics}

The state process of the problem is the pair $\Xi_t=(X_t,\mu_t)
\in \sX = \cX\times\Sigma^{d-1}_N$, 
consisting of the position of the tagged player $X_t \in \cX$
and the fractions of other players across states 
$\mu_t \in \Sigma^{d-1}_N$.
More precisely, for given fixed (uncontrolled) rates 
$$
(\lambda^0,\lambda^1):\sX\mapsto\R_+^{d-1}\times\R_+^{d-1}
$$
and \emph{feedback controls}
$$
(\alpha,\beta):[0,T]\times\sX\mapsto\R_+^{d-1} 
\times\R_+^{d-1},
$$
the \emph{controlled} state process 
$\Xi_t=(X_t,\mu_t) \in \sX$ is a 
time-inhomogeneous Markov chain with an infinitesimal generator
given by
$$
\sL^{\alpha, \beta} := \sL^{\alpha}_x + \sL^{\beta}_\mu.
$$
Here, using the rate convention introduced in the notation
part of the Introduction,
for any $\varphi:\sX\mapsto\R$ and 
$(t,x,\mu)\in[0,T]\times \sX$ with $\mu=(n^\mu_1,\ldots,n^\mu_d)/N$, 
\begin{align*}
(\sL^{\alpha}_x\varphi)(t,x,\mu) &:= \sum_{y\neq x} \left(\lambda^0_y(x,\mu)+\lambda^1_y(x,\mu)\alpha_y(t,x,\mu)\right)[\varphi(y,\mu)-\varphi(x,\mu)], \\
(\sL^{\beta}_\mu\varphi)(t,x,\mu) &:= \sum_{y\neq z} n^\mu_z \left(\lambda^0_y(z,\mu+e_{x,z})+\lambda^1_y(z,\mu+e_{x,z})
\beta_y(t,z,\mu+e_{x,z})\right)[\varphi(x,\mu+e_{y,z})-\varphi(x,\mu)],
\end{align*}
where we recall that $e_x\in\{0,1\}^d$ denotes the $x$-th unit vector and $e_{x,z} := \frac1N(e_x-e_z)$ is interpreted as a vector which adds the position of the tagged player in state $x\in\cX$ and
subtracts the position of a generic untagged player in state
$z\in\cX$.  
Equivalently, the transition rates of 
$\Xi_t=(X_t,\mu_t)$ are given by
\begin{align*}
\text{rate from}\ (x,\mu)\ \text{to}\ (x',\mu)& =
\lambda^0_{x'}(x,\mu)+\lambda^1_{x'}(x,\mu)\alpha_{x'}(t,x,\mu),
\quad x\neq x',\\
\text{rate from}\ (x,\mu)\ \text{to} \ (x,\mu+e_{y,z})& =
n^\mu_z \, [\lambda^0_y(z,\mu+e_{x,z})+\lambda^1_y(z,\mu+e_{x,z})
\beta_y(t,z,\mu + e_{x,z})],\quad z\neq y.
\end{align*}
Finally, we let $\P^{\alpha\otimes\beta}$ be the probability
measure induced by this Markov chain.

\subsection{Control problem of the tagged player}
\label{ss:control}
We assume that  all players use the same running and terminal cost functions 
$$
\ell: \sX \times\R_+^{d-1}\mapsto\R_+,\qquad g:\sX\mapsto\R_+.
$$
Throughout the paper, and sometimes implicitly,
we assume that $\ell$ is strongly convex:

\begin{Asm}
\label{asm:cost}
The running cost $\ell$ is continuous and there exists $\gamma>0$ such that, for all 
$(x,\mu,a,a')\in\sX\times\R_+^{d-1}\times\R_+^{d-1}$,
$$
\tau \ell(x,\mu, a')+(1-\tau)\ell(x,\mu, a)
- \ell(x,\mu,\tau a' +(1-\tau)a)\geq 
\gamma \tau (1-\tau) |a'-a|_2^2, \qquad \forall \tau \in [0,1].
$$
\end{Asm}
\noindent
We continue to define the \emph{admissible controls}.

\begin{Def}
\label{def:admissible}
{\rm{Any Lipschitz continuous function $\alpha:[0,T]\times\sX\mapsto \R_+^{d-1}$
is  called an}} admissible control {\rm{and  $\sA$ is  
the set of all admissible controls. 
For $\alpha\in\sA$, $t \in [0,T]$, $y\neq x \in \cX$, and
$\mu \in \Sigma_N^{d-1}$,
the rate $\alpha_y(t,x,\mu)$ is defined  as in
in the notation part of the Introduction.}}
\end{Def}

For a given admissible control $\beta\in\sA$ that every other player uses, 
the tagged player solves the following stochastic optimal control problem
$$
v^\beta(t,x,\mu):=\inf_{\alpha\in\sA}\ J(t,x,\mu,\alpha;\beta),\quad (t,x,\mu)\in[0,T]\times\sX,
$$
where the tagged player's \emph{cost functional} is given by
\begin{equation}\label{eq:payoff} 
J(t,x,\mu,\alpha;\beta) := \E^{\alpha\otimes\beta}\Bigl[ \int_t^T \ell(\Xi_u,\alpha(u,\Xi_u))\,\d u + g(\Xi_T) \,|\, \Xi_t=(x,\mu)\Bigr],\quad (t,x,\mu)\in[0,T]\times\sX,
\end{equation}
and $\E^{\alpha\otimes\beta}[\,\cdot\,]$ denotes the expectation with respect to the measure $\P^{\alpha\otimes\beta}$.

For any $\varphi:\cX\mapsto\R$ and $x \in \cX$, define
the first difference vector of $\varphi$ at $x$ by
$$
\Delta_x \varphi:=(\varphi(y)-\varphi(x))_{y\in\cX\setminus\{x\}} \ \in\ \R^{d-1}.
$$
It is classical that the value function $v^\beta$ solves the following dynamic programming equation
\cite{FS},
\begin{equation}\label{eq:HJB-beta}\tag{HJB$(\beta)$}
\left\{
\begin{aligned}
-\frac{\d}{\d t} v^\beta(t,x,\mu) &= H(x,\mu,\Delta_x v^\beta(t,\cdot,\mu)) 
+ \sL^\beta_\mu v^\beta(t)(t,x,\mu),\\
v^\beta(T,x,\mu)&=g(x,\mu),
\end{aligned}
\right.
\end{equation}
for $(t,x,\mu)\in [0,T]\times\sX$, where  the \emph{Hamiltonian} $H$ is given by,
$$
H(x,\mu,p) := \inf_{a\in\R_+^{d-1}} \ \Bigl\{\ell(x,\mu,a) + \sum_{y\neq x}(\lambda^0_y(x,\mu)+\lambda^1_y(x,\mu)a_y)\,p_y\Bigr\},\quad (x,\mu,p)\in\sX\times\R^{d-1}.
$$
Due to the strong convexity of $\ell$ in Assumption~\ref{asm:cost}, 
there is a unique minimizer $\hat\alpha:\sX\times\R^{d-1}\mapsto\R_+^{d-1}$
given by
$$
\hat\alpha(x,\mu,p):= \underset{a\in\R_+^{d-1}}
{\text{arg\,min}} \ \Bigl\{ \ell(x,\mu,a) 
+ \sum_{y\neq x}a_y\lambda^1_y(x,\mu) p_y \Bigr\},
\qquad (x,\mu,p)\in \sX \times \R^{d-1}.
$$

Since $\lambda^1$ is a function on a finite set,
it is bounded.  This bound,  the conjugate correspondence theorem, and the $\gamma$-strong convexity assumption on $\ell(x,\mu,\cdot)$ 
directly imply the following Lipschitz estimate.

\begin{Lem}\label{lem:control-lip}
There is a constant $c^*_a>0$
such that 
$$
|\hat\alpha(x,\mu,p)-\hat\alpha(x,\mu,p')|_2 \leq  c^*_a\  |p-p'|_2,
\qquad  (x,\mu,p,p')\in\sX\times\R^{d-1}\times\R^{d-1}. 
$$
\end{Lem}

\subsection{Markov perfect equilibria and the NLL equation}
\label{ss:NLL}

Consider the following procedure: for $\beta\in\sA$, solve \eqref{eq:HJB-beta} to obtain the unique classical solution $v^\beta$. Then, the map
\begin{equation}
\label{eq:alpha*}
\alpha^*(\beta)(t,x,\mu) := \hat\alpha(x,\mu, \Delta_x v^\beta (t,\cdot,\mu)),\quad (t,x,\mu)\in[0,T]\times\sX,
\end{equation}
is the unique optimal feedback control in $\sA$. This defines a functional $\beta\mapsto\alpha^*(\beta)$.

\begin{Def}
\label{def:MPE}
{\rm{We call a feedback function $\alpha\in\sA$ a}}
symmetric Markov perfect equilibrium {\rm{(MPE) 
for the $N+1$ player game if $\alpha = \alpha^*(\alpha)$.}}
\end{Def}

Clearly, a feedback function $\alpha\in\sA$ is a MPE if, and only if, 
$$
J(t,x,\mu,\alpha;\alpha)\leq J(t,x,\mu,\tilde \alpha;\alpha),\qquad \forall (t,x,\mu,\tilde \alpha)\in[0,T]\times\sX\times\sA.
$$

The Nash-Lasry-Lions equation for the $N+1$ player game,
referred to as the $N$-NLL equation to emphasize its dependence on
the number of players,  is defined as the following (backward) system of ODEs:
\begin{equation}
\label{eq:N-NLL}\tag{$N$-NLL}
\begin{cases}
 &-\frac{\d}{\d t} v(t,x,\mu) = H(x,\mu,\Delta_x v(t,\cdot,\mu))+\sL^{\alpha}_\mu v(t,\cdot)(t,x,\mu),\\
 &\alpha(t,x,\mu) = \hat\alpha(x,\mu,\Delta_xv(t,\cdot,\mu)),\\
 &v(T,x,\mu)=g(x,\mu),
 \end{cases}
\end{equation}
for $(t,x,\mu)\in[0,T]\times\sX$. Gomes, Mohr and Souza~\cite[Theorem 6]{GMS1}  proved the following result.

\begin{Prop}
The \eqref{eq:N-NLL} equation admits a unique classical solution $v$
and  the function
$$
\alpha(t,x,\mu) := \hat\alpha(x,\mu,\Delta_xv(t,\cdot,\mu)),\quad (t,x,\mu)\in[0,T]\times\sX,
$$
is the unique Markov perfect equilibrium of the $N+1$ player game.
\end{Prop}

\section{Picard-type iterations}
\label{sec:iterations}
Recall that the tagged
player's value function $v^\beta$ is the 
unique classical solution to \eqref{eq:HJB-beta}
corresponding to the untagged player control $\beta$,
and the best response $\alpha^*(\beta)$ of the tagged player is defined by \eqref{eq:alpha*}.
For $\rho \in [0,1)$,
starting from a fixed initial control $\beta^{(0)}\in\sA^*$,
for $n\ge 1$, we recursively define
$v^{(n)}:=v^{\beta^{(n-1)}}$ and 
\begin{equation}\label{eq:control-seq}
\alpha^{(n)}(t,x,\mu) :=\alpha^*(\beta^{(n-1)})(t,x,\mu)
= \hat \alpha (x,\mu,\Delta_x v^{(n)}(t,\cdot,\mu)),
\quad (t,x,\mu)\in[0,T]\times\sX,
\end{equation}
along with the untagged players' updated control
$$
\beta^{(n)} := \we \beta^{(n-1)} + (1-\we)\alpha^{(n)}, 
\qquad n \ge 1.
$$ 

\begin{Def}
\label{def:iterations}
{\rm{The iterative scheme with $\we=0$ is called the}}
Picard algorithm, {\rm{and the 
one with  $\we\in(0,1)$ the}} weighted Picard algorithm.
\end{Def}

Our main convergence result is the following.
When necessary to emphasize the dependence on the parameter $\we$,  we 
sometimes write $v^{(n)}_\we$ for the corresponding value functions.

When necessary to emphasize the dependence on the parameter $\we$,  we 
sometimes write $v^{(n)}_\we$ for the corresponding value functions.

\begin{Thm}
\label{thm:convergence}
For any $\we \in [0,1)$ and any initial $\beta^{(0)}\in\sA^*$, the sequence of value function $(v_\we^{(n)})_{n\geq1}$ converges uniformly 
to the unique solution $v$ of the
\eqref{eq:N-NLL} equation, and the sequence of controls $(\alpha_\we^{(n)})_{n\geq1}$ 
converges uniformly to the unique Markov perfect equilibrium given by
$\alpha(t,x,\mu) = \hat\alpha(x,\mu,\Delta_xv(t,\cdot,\mu))$, for $(t,x,\mu)\in[0,T]\times\sX$. 
\end{Thm}

We prove this theorem in Section~\ref{sec:proofs} below. 
As the proof  is more transparent
for $\we=0$, we first analyze this case
and consider the general case later in subsection~\ref{ss:LPicard}.
We also have the following exponential convergence rates
proved in the subsection \ref{ss:rate} below.

\begin{Prop}
\label{prop:rate}
For $ \we \in(0,1)$, there are $\gl \in (0,1)$ and  $c_\we>0$ satisfying,
$$
\sup_{t \in [0,T]} |v^{(n)}_\we(t)-v(t)|_2  \le c_\we  \gl^n, \qquad n \ge1,
$$
where $| \cdot|_2$ is as in \eqref{eq:norm}. The thresholds
$\gl$ are non-decreasing in $\we$ with $\lim_{\we \downarrow 0} \gl=0$.
Additionally, for every $\gamma>0$ there exists $\hat{c}_\gamma>0$
such that
$$
\sup_{t \in [0,T]} |v_{0}^{(n)}(t)-v(t)|_2  \le \hat{c}_\gamma \gamma^n, \qquad n \ge1,
$$
\end{Prop}

\begin{Rmk}{\rm{Theorem \ref{thm:convergence} might seem to contradict
the intuition developed for stationary or dynamic games in discrete time.
In these models, the convergence of the Picard-type iterations is not always expected,
and fictitious play and its generalizations were developed
to address this difficulty.   However,  the strong convexity 
Assumption~\ref{asm:cost} and the dynamic momentum inherent
in  continuous-time models, sets our model apart from the general 
discrete-time settings. These features  are the underlying reasons for the convergence 
proved in Theorem~\ref{thm:convergence}. 
In the structure studied here, players are not allowed to 
instantly move their state process to different state with probability
one; this would be possible only with infinite transition rates but large transition rates incur high costs.  A recent 
paper~\cite{HSY}  shows that time discretizations of the models 
we study have unique MPE only when the time
step $h$ is sufficiently small. The  continuous-time case, formally
corresponding to the limit $h$ approaching to zero, 
also has this uniqueness property~\cite{GMS1}. 
These  results highlight the difference between models with small 
$h$ or in continuous-time models and general discrete-time ones.  
In particular, the stiffness of the structure  we study 
prevents the oscillations that might emerge in discrete time.}}
\end{Rmk}

\section{Convergence proof}
\label{sec:proofs}
In this section, we prove Theorem~\ref{thm:convergence} and Proposition \ref{prop:rate}.
We start with  proving several technical point-wise estimates.
The case $\we=0$ is considered in Subsection~\ref{ss:Picard},
the general case in Subsection~\ref{ss:LPicard}
and the convergence rate is proved in Subsection~\ref{ss:rate}.

\subsection{Estimates}
\label{ss:bounds}

In this subsection, we prove several
estimates that are used in the convergence proofs with
constants depending on  all given functions and $T$, as well
as $N$ and $d$.  As $\sX$ is a finite set, we do not need
Lipschitz-type estimates for bounded functions 
on $\sX$, albeit the bounds on the differences deteriorate
as $N$ or $d$ gets larger even with uniform point-wise bounds.
We omit these dependencies in our notation. 

Recall  the best response function $\alpha^*(\beta)$ of
\eqref{eq:alpha*}, and that the value function $v^\beta$ of 
the tagged player, when other players use $\beta$,
solves the dynamic programming equation  \eqref{eq:HJB-beta}.

\begin{Lem}
\label{lem:constants}
There exist constants $c_v, c_a >0$ such that 
$$
\sup_{\beta\in\sA} \ \max_{(t,x,\mu)\in[0,T]\times\sX} \ |v^\beta(t,x,\mu)|\leq c_v,
\quad \sup_{\beta\in\sA} \ \max_{(t,x,\mu)\in[0,T]\times\sX} \  |\alpha^*(\beta)(t,x,\mu)|_2\leq c_a.
$$
\end{Lem}

\begin{proof}
Since $\ell,g$ are assumed to be non-negative, 
for any $\beta\in\sA$ the value function $v^\beta$ of 
is also non-negative. 
Then, by choosing $\alpha\equiv0$ we obtain the following uniform upper bound, 
\begin{align*}
 0 &\leq \sup_{\beta\in\sA} \ \max_{(t,x,\mu)\in[0,T]\times\sX} 
 \ v^\beta(t,x,\mu) \leq 
 \sup_{\beta\in\sA} \ \max_{(t,x,\mu)\in[0,T]\times\sX} 
 \ J(t,x,\mu,0;\beta)
 \\
 &\le \max_{(x,\mu)\in\sX} \ (T\ell(x,\mu,0) + g(x,\mu)) =: c_v<\infty.   
\end{align*}
In view of Lemma~\ref{lem:control-lip},
$$
|\hat\alpha(x,\mu,\Delta_x v^\beta(t,\cdot,\mu))|_2
\le |\hat\alpha(x,\mu,0)|_2 +c^*_a \ |\Delta_x v^\beta(t,\cdot,\mu)|_2
\le |\hat\alpha(x,\mu,0)|_2 + 2 \sqrt{d-1}\  c^*_a c_v.
$$
Once again we use the fact that $\hat\alpha (\cdot,\cdot,0)$ as a
function on a finite set is bounded.  Consequently, the claimed
estimate holds with the constant $c_a$ defined by
\begin{equation*}
c_a:= \max_{(x,\mu)\in\sX} \hat\alpha(x,\mu,0) + 2\|\lambda^1\|_\infty c_v\sqrt{d-1}/\gamma<\infty.
\qedhere
\end{equation*}
\end{proof}

In view of the above estimate on the best response function,
we conclude that all MPEs must satisfy this point-wise estimate.
Hence, we may restrict the controls to 
the ones which are bounded by the constant $c_a$ 
constructed in the above lemma, i.e., we set
\begin{equation}
\label{eq:smallA}
\sA^*:= \bigl\{ \alpha \in \sA\ :\  \sup_{t \in [0,T]}\ |\alpha(t)|_2 \le c_a \bigr\},
\end{equation}
where we used the notation $\alpha(t)$ introduced in the notation subsection
of the Introduction.
We summarize these observations in the next corollary.
\begin{Cor}
\label{cor:H}
Let $c_a$ and $\sA^*$ be as above.  Then,
for any $\beta \in \sA$ and $(t,x,\mu)\in[0,T]\times\sX$,
the tagged player's value function $v^\beta$  satisfies
$$
    H(x,\mu,\Delta_x v^\beta(t,\cdot,\mu)) =
    \inf_{a\in\R_+^{d-1}, |a|_2 \le c_a} \ \Bigl\{\ell(x,\mu,a) + \sum_{y\neq x}(\lambda^0_y(x,\mu)+\lambda^1_y(x,\mu)a_y)\ 
    (\Delta_x v^\beta(t,\cdot,\mu))_y\Bigr\}.
$$
\end{Cor}

We emphasize that the constant $c_a$ depends on all of the given functions
but is otherwise independent of the untagged player's control $\beta$.
The next two estimates play a central role
in our subsequent analysis.

\begin{Lem}\label{lem:alpha-leq-value}
There exists a constant $c_0>0$ such that for 
any $\beta, \tilde\beta \in \sA$, 
$$
|\alpha^*(\beta)(t)- \alpha^*(\tilde  \beta)(t)|_2\leq 
c_0 \  |v^{\beta}(t)- v^{\tilde \beta}(t)|_2 ,
\qquad t\in[0,T].
$$
\end{Lem}
\begin{proof}
To simplify the notation, we fix $\beta, \tilde\beta \in \sA$
and set $\alpha= \alpha^*(\beta)$,
$\tilde \alpha = \alpha^*(\tilde \beta)$,
$v=v^{\beta}$ and 
$\tilde v = v^{\tilde \beta}$.
Using Lemma \ref{lem:control-lip}, for $t\in[0,T]$, we directly estimate:
\begin{align*}
|\tilde \alpha(t)- \alpha(t)|_2^2 &= \sum_{(x,\mu)\in\sX}|\hat\alpha(x,\mu,\Delta_x \tilde v(t,\cdot,\mu))-\hat\alpha(x,\mu,\Delta_x v(t,\cdot,\mu))|_2^2 \\
&\leq (c^*_a)^2\ \sum_{(x,\mu)\in\sX}|\Delta_x \tilde v(t,\cdot,\mu)-\Delta_x v(t,\cdot,\mu)|_2^2\\
&\leq 2 (c^*_a)^2\ \sum_{(x,\mu)\in\sX} \Bigl(\sum_{y\neq x}|\tilde v(t,y,\mu)-v(t,y,\mu)|^2 + (d-1)|\tilde v(t,x,\mu)-v(t,x,\mu)|^2\Bigr)\\
&= 4(c^*_a)^2 (d-1)\  \sum_{(x,\mu)\in\sX} |\tilde v(t,x,\mu)-v(t,x,\mu)|^2=4(c^*_a)^2 (d-1)\ |\tilde v(t)-v(t)|_2^2.
\end{align*} 
Setting $c_0:=2c_a^*\sqrt{d-1}$ completes the proof
\end{proof}

\begin{Lem}\label{lem:value-leq-beta}
There exists a constant $c^*>0$ such that 
any $\beta, \tilde\beta \in \sA^*$,
$$
|v^{\beta}(t)- v^{\tilde \beta}(t)|^2_2
\leq c^*\ \int_t^T |\tilde\beta(s)-\beta(s)|_2^2\ \d s,
\qquad t\in[0,T].
$$
\end{Lem}
\begin{proof}
We again fix $\beta, \tilde\beta \in \sA^*$ and set
$\alpha= \alpha^*(\beta)$,
$\tilde \alpha = \alpha^*(\tilde \beta)$,
$v=v^{\beta}$, 
$\tilde v = v^{\tilde \beta}$, 
$$
w := \tilde v -v, \quad \text{and}\quad 
m(t):=|w(t)|_2^2,\quad t\in[0,T].
$$

We proceed in several steps.
\vspace{5pt}

\noindent
\emph{Step 1.} For $t\in[0,T]$, we use the dynamic programming equation \eqref{eq:HJB-beta} to compute
\begin{equation}\label{eq:m(t)}
- \frac{\d}{\d t} m(t) = -\sum_{(x,\mu)\in\sX} 
\frac{\d}{\d t} |w(t,x,\mu)|^2
=-\sum_{(x,\mu)\in\sX} \Bigl(\frac{\d}{\d t} w(t,x,\mu)\Bigr) \, 
w(t,x,\mu)= \cI + \cJ + \cK,
\end{equation}
where, for $(t,x,\mu)\in[0,T]\times\sX$, 
\begin{align*}
\cI &:= \sum_{(x,\mu)\in\sX} (H(x,\mu,\Delta_x\tilde v(t,\cdot,\mu)) - H(x,\mu,\Delta_xv(t,\cdot,\mu))) \, w(t,x,\mu),\\
\cJ &:= \sum_{(x,\mu)\in\sX} (\sL_\mu^{\tilde \beta}-\sL_\mu^{\beta}) \tilde v(t)(t,x,\mu)\,w(t,x,\mu),\\
\cK& : = \sum_{(x,\mu)\in\sX} \sL_\mu^\beta w(t)(t,x,\mu) \, w(t,x,\mu).
\end{align*}
We continue by estimating $\cI,\cJ$, and $\cK$.
\vspace{5pt}

\noindent
\emph{Step 2.} We start with $\cI$.  For any $(t,x,\mu)\in[0,T]\times\sX$, 
by the Cauchy-Schwarz and Jensen inequalities, and Corollary~\ref{cor:H}, 
\begin{align*}
\sum_{(x,\mu)\in\sX}|H(x,\mu,\Delta_x\tilde v(t,\cdot,\mu)) 
&- H(x,\mu,\Delta_xv(t,\cdot,\mu))|^2 \\
&\leq\sum_{(x,\mu)\in\sX} \ \sup_{a\in\R_+^{d-1},\ 
|a|_2\leq c_a } \ \Big( \sum_{y\neq x} \lambda^1_y(x,\mu)a_y |(\Delta_xw(t,\cdot,\mu))_y| \Big)^2\\
&\leq(d-1)  c_a^2\|\lambda^1\|_\infty^2\ 
 \sum_{(x,\mu)\in\sX} \sum_{y\neq x} |(\Delta_xw(t,\cdot,\mu))_y|^2  \\
&\leq c_1^2 \sum_{(y,\mu)\in\sX} |w(t,y,\mu)|^2 = c_1^2 \ m(t),
\end{align*}
for some constant $c_1$. Hence,
$$
| \cI | \le 
\Big(\sum_{(x,\mu)\in\sX}|H(x,\mu,\Delta_x\tilde v(t,\cdot,\mu)) 
- H(x,\mu,\Delta_xv(t,\cdot,\mu))|^2 \Big)^{1/2} \
\sqrt{m(t)}
\le c_1\  m(t).
$$

Using Young's inequality repeatedly, we estimate $\cJ$ by
\begin{align*}
2 | \cJ | &\leq   \sum_{(x,\mu)\in\sX} [(\sL_\mu^{\tilde\beta}-\sL_\mu^\beta) \tilde v(t)(t,x,\mu)|^2
+ \ m(t) \\
&\leq  \sum_{(x,\mu)\in\sX} \Bigl(
\sum_{y\neq z} n^\mu_z 
\lambda^1_y(z,\mu+e_{xz})
(\tilde \beta_y-\beta_y)(t,z,\mu+e_{xz})(\tilde v(t,x,\mu+e_{yz})-\tilde v(t,x,\mu))\Bigr)^2
+ \ m(t)\\
&\leq  4 N^2d c_v^2\|\lambda^1\|_\infty^2 \sum_{(x,\mu)\in\sX} 
\sum_{y\neq z} [(\tilde \beta_y-\beta_y)(t,z,\mu+e_{xz})]^2 +  m(t)\\
&= 2c_2 |\tilde\beta(t)-\beta(t)|_2^2 + m(t),
\end{align*}
where $c_2=2N^2d c_v^2\|\lambda^1\|_\infty^2 $.

Since  $\beta,\tilde\beta\in\sA^*$, they are point-wise bounded by $c_a$, and consequently there exists $c_3>0$ such that
\begin{align*}
|\cK| &\leq\!\!\!\sum_{(x,\mu)\in\sX, y\neq z} \!\!\!\!
n^\mu_z [\lambda^0_y(z,\mu+e_{xz})+\lambda^1_y(z,\mu+e_{xz})\beta_y(t,z,\mu + e_{xz})]
|w(t,x,\mu+e_{yz})-w(t,x,\mu)| |w(t,x,\mu)|\\
&\leq c_3 m(t).   
\end{align*}

\noindent
\emph{Step 3.} Combining  the estimates
for $\cI,\cJ,\cK$ obtained in the previous step, we arrive at 
$$
-\frac{\d}{\d t} m(t) \leq \cI + \cJ + \cK  
\leq  c_4(m(t)+|\tilde\beta(t)-\beta(t)|_2^2)
$$
where $c_4= \max\{c_1+c_3+1/2,c_2\}$. 
By Gr\"{o}nwall's inequality, this implies that 
$$
0\leq m(t_0) \leq e^{c_4(t_1-t_0)}m(t_1)+\int_{t_0}^{t_1}e^{c_4(s-t_0)}|\tilde\beta(s)-\beta(s)|_2^2\,\d s,\quad 0\leq t_0\leq t_1\leq T.
$$
Note that $m(T)=0$. Hence, applying this with $t_0=t$ and $t_1=T$ we obtain,
\begin{equation*}
0\leq m(t) \leq  c^* \int_t^T |\tilde\beta(s)-\beta(s)|_2^2\,\d s,\qquad 0\leq t\leq T,
\end{equation*}
where $c^*:=\exp(c_4T)$.
\end{proof}

\subsection{Picard algorithm}
\label{ss:Picard}

In this subsection, we analyze the  pure
Picard iteration corresponding to $\we=0$.
In this case, given $\beta^{(0)}\in\sA^*$, 
the algorithm reduces to the following definitions:
$$
    v^{(n)} := v^{\beta^{(n-1)}},\quad 
    \alpha^{(n)} :=\alpha^*(\beta^{(n-1)}), 
    \quad \beta^{(n)} := \alpha^{(n)},\quad n\geq1.
$$

\noindent
\emph{Proof of Theorem~\ref{thm:convergence} with $\we=0$}.  Set
\begin{equation}
\label{eq:wm}
w^{(n)} := v^{(n+1)}-v^{(n)},\qquad 
m^{(n)}(t):= |w^{(n)}(t)|_2^2,\qquad t\in[0,T],\ n\geq1. 
\end{equation}
Since $\beta^{(0)} \in \sA^*$, by Corollary~\ref{cor:H}
we can conclude that all  
$\alpha^{(n)}, \beta^{(n)} \in \sA^*$.
Then, by Lemma~\ref{lem:value-leq-beta}, there exists  $c^*$ satisfying,
$$
m^{(n+1)} (t) \leq c^* \int_t^T |\beta^{(n+1)}(s)-\beta^{(n)}(s)|_2^2\,\d s,
\qquad t\in[0,T],\ n\geq1.
$$
In view of Lemma~\ref{lem:alpha-leq-value}, there is
$c_0$ such that
\begin{align*}
|\beta^{(n+1)}(t)-\beta^{(n)}(t)|_2^2& = |\alpha^*(\beta^{(n)})(t)-\alpha^*(\beta^{(n-1)})(t)|_2^2\\
&\leq c_0 |v^{\beta^{(n)}}(t)-v^{\beta^{(n-1)}}(t)|_2^2\\
&= c_0 |v^{(n+1)}(t)-v^{(n)}(t)|_2^2,
\qquad t\in[0,T],\ n \geq 1,
\end{align*}
Combining these results yields
\begin{equation}
\label{eq:est}
m^{(n+1)}(t)\leq c^* c_0 \int_t^T m^{(n)}(s)\,\d s,
\qquad t\in[0,T],\ n\geq 1.
\end{equation}

We now claim that 
\begin{equation}
\label{eq:induction}
m^{(n)}(t) \leq  \hat{c}\, \frac{( \hat{c}\, (T-t) )^n}{n!}, \qquad t\in[0,T],\ n\geq1,
\end{equation}
for some constant $\hat{c}$.  Indeed, since $\beta^{(n)} \in \sA^*$,
$|\beta^{(n)}(t)|_2 \le c_a$ for each $n\ge 0$, we have
$$
m^{(1)}(t) \le c^* \int_t^T |\beta^{(1)}(s) - \beta^{(0)}(s)|_2^2\ \d s
\le  c^* \int_t^T (2c_a)^2\ \d s = 4 c^* c_a^2 \ (T-t),\qquad t\in[0,T].
$$
We continue by proving \eqref{eq:induction} by induction with 
$\hat{c}=  \max\{ c^* c_0 \, ,\,   c^* c_a^2\}$.
Suppose that it holds for some $n\ge 1$. Then,  by
\eqref{eq:est},
$$
m^{n+1}(t) \leq c^* c_0 
\int_t^T m^{(n)}(s)\,\d s \leq c^* c_0 \hat{c} 
\int_t^T \frac{( \hat{c}\, (T-s) )^n}{n!}\,\d s 
=c^* c_0 \frac{(\hat c(T-t))^{n+1}}{(n+1)!}
\le \hat{c} \frac{(\hat{c}(T-t))^{n+1}}{(n+1)!}.
$$
The definition of $m(t)$ and \eqref{eq:induction}
directly imply that
$$
\sup_{t\in[0,T]}
\Bigl(\sum_{(x,\mu)\in\sX}|(v^{(n+1)}-v^{(n)})(t,x,\mu)|^2\Bigr)^{1/2}
= \sup_{t\in[0,T]} \big(m^{(n)}(t)\big)^{1/2}
\leq 
(\hat{c})^{1/2}  \left(\frac{(\hat{c}\,T )^n}{n!}\right)^{1/2} =: a_n,\quad \ n\geq1.
$$
Recall the Banach space $\cS^k$ from \eqref{eq:snorm}.
Since $\sum_n a_n<\infty$, it follows that $(v^{(n)})_{n\geq1}$ is a Cauchy sequence in $\cS^1$. 
Hence, there exists $\bar v\in \cS^1$ such that $v_n\to \bar v$ in $\cS^1$. Using Lemma~\ref{lem:control-lip}, this implies that $\alpha_n\to\bar \alpha$ in $\cS^{d-1}$, where
$$
\bar \alpha (t,x,\mu) := \hat \alpha(x,\mu,\Delta_x \bar v(t,\cdot,\mu)),\quad (t,x,\mu)\in[0,T]\times\sX.
$$
It remains to argue $\bar v$ satisfies \eqref{eq:N-NLL}. Clearly, $\bar v(T,\cdot) = \lim_{n\to\infty} v^{(n)}(T,\cdot)=g$. Next, we pass to the limit in the following equation
$$
 -\frac{\d}{\d t} v^{(n)}(t,x,\mu) = H(x,\mu,\Delta_x v^{(n)}(t,\cdot,\mu)) + \sL^{\alpha^{(n-1)}}_\mu v^{(n)}(t)(t,x,\mu) =: F(x,\mu,\alpha^{(n-1)}(t),v^{(n)}(t)).
$$
For any compactly supported smooth test function $\phi\in C^\infty((0,T))$, integration by parts and the dominated convergence theorem imply, for any $(x,\mu)\in\sX$,  
\begin{align*}
\lim_{n\to\infty} \int_0^T \frac{\d}{\d t} v^{(n)}(t,x,\mu)\phi(t)\,\d t 
= - \lim_{n\to\infty}  \int_0^T v^{(n)}(t,x,\mu)\phi'(t)\,\d t 
= - \int_0^T \bar v(t,x,\mu)\phi'(t)\,\d t.  
\end{align*}
Additionally,
\begin{align*}
\lim_{n\to\infty}  \int_0^T \frac{\d}{\d t} v^{(n)}(t,x,\mu)\phi(t)\,\d t &= -
\lim_{n\to\infty} \int_0^T F(x,\mu,\alpha^{(n-1)}(t),v^{(n)}(t))\phi(t)\,\d t \\
&= \int_0^T F(x,\mu,\bar \alpha(t),\bar v(t))\phi(t)\,\d t.
\end{align*}
Hence, $-F(x,\mu,\bar \alpha(t),\bar v(t))$ is the weak (in the sense of distributions) derivative of $\bar v(t,x,\mu)$ on $(0,T)$. But the former is continuous in $t$, so that it is the strong derivative of $\bar v(t,x,\mu)$ on $(0,T)$. This shows that $\bar v$ indeed satisfies the NLL equation \eqref{eq:N-NLL}, which has a unique classical solution implying $\bar v=v$.
\qed

\subsection{Weighted Picard algorithm}
\label{ss:LPicard}

We fix $0<\we<1$  and omit the dependence on $\we$ in our notation. 
Starting with an arbitrary $\beta^{(0)}\in\sA^*$,
the defining equations are given by
$$
v^{(n)} = v^{\beta^{(n-1)}},\quad \alpha^{(n)}=\alpha^*(\beta^{(n-1)}), \quad \beta^{(n)} = \we\beta^{(n-1)}+(1-\we)\alpha^{(n)},\quad n\geq1.
$$
As all best responses are in the convex set $\sA^*$, we conclude that
$\beta^{(n)}, \alpha^{(n)}$ are also in $\sA^*$.  In particular, they are
bounded by the constant $c_a$ of Lemma~\ref{lem:constants}.
\vspace{5pt}

\begin{proof}[Proof of Theorem~\ref{thm:convergence} 
with $\we \in (0,1)$.]
Let $w^{(n)}, m^{(n)}(t)$ be as in \eqref{eq:wm} and set
$$
\gamma^{(n)}:=\alpha^{(n+1)}-\beta^{(n)},\quad 
\Gamma^{(n)}(t) :=|\gamma^{(n)}(t)|_2^2,
\qquad t\in[0,T],\ n\geq 0.
$$

\noindent
\emph{Step 1.} In this step,
we derive recursions for $(m^{(n)},\Gamma^{(n)})$ which we later solve in the next steps. 
We first observe that
$$
\beta^{(n+1)}-\beta^{(n)} = \we\beta^{(n)} +(1-\we)\alpha^{(n+1)} - \beta^{(n)} = (1-\we) \gamma^{(n)}.
$$
Hence,  by Lemma~\ref{lem:value-leq-beta},
\begin{equation}\label{eq:mn-recursion}
m^{(n+1)} (t) \leq c^* \int_t^T |\beta^{(n+1)}(s)-\beta^{(n)}(s)|_2^2\,\d s
=c^*(1-\we)\int_t^T\Gamma^{(n)}(s)\,\d s,\quad t\in[0,T],\ n\geq0.
\end{equation}
In view of the updating rule for $\beta^{(n)}$,
$$
\gamma^{(n)} = \alpha^{(n+1)}-\we\beta^{(n-1)}-(1-\we)\alpha^{(n)} 
= (\alpha^{(n+1)}-\alpha^{(n)}) + \we\gamma^{(n-1)},\quad n\geq1.
$$
We use the above with Lemma~\ref{lem:alpha-leq-value} 
to obtain the following recursion for $\Gamma^{(n)}$ with $\delta>0$ 
to be chosen in the next step,
\begin{align}
\nonumber
\Gamma^{(n)}(t) &\leq  |\alpha^{(n+1)}(t)-\alpha^{(n)}(t)|_2^2 + 
2 \we | (\alpha^{(n+1)}(t)-\alpha^{(n)}(t)) \gamma^{(n-1)}(t)|_2^2
+ \we^2 \, \Gamma^{(n-1)}(t)\\
\nonumber
&\le (1+\delta^{-1}) |\alpha^{(n+1)}(t)-\alpha^{(n)}(t)|_2^2 + (1+\delta)\we^2 \,\Gamma^{(n-1)}(t)\\
\label{eq:Gamma-recursion}
&\leq c_0(1+\delta^{-1}) m^{(n)}(t) + (1+\delta)\we^2 \,\Gamma^{(n-1)}(t),
\qquad n\geq1.
\end{align}

\noindent
\emph{Step 2.} Set  $\we_\delta:=(1+\delta)\we^2$, and choose
$$
\delta=\delta_\we= \frac{1-\we^2}{2\we^2} 
\qquad \Rightarrow \qquad
\we_\delta= 1- \frac12 (1-\we^2) <1.
$$
As $|\alpha^{(n)}(t)|_2,  |\beta^{(n)}(t)|_2 \le c_a$, we have $\Gamma^{(n)} \le 4 c_a^2$.
We claim that 
$$
\Gamma^{(0)}(t) \le 4c_a^2, \quad \text{and}\quad
\Gamma^{(n)}(t)\leq c_0(1+\delta_\we^{-1}) \ \sum_{k=0}^{n-1} 
\we_\delta^k \, m^{(n-k)}(t) 
+ 4 c_a^2  \we_\delta^{n},\quad t\in[0,T],\ n\geq1.
$$
%Indeed, we directly verify the above inequality for $n=1$. 
%Using the convention that the sum over the empty set is zero, for $n=0$,
%the right hand side is equal to $4c_a^2$ and we have discussed that it  
%bounds all $\Gamma^{(n)}$. 
Indeed, for $n=1$ the claimed estimate is exactly \eqref{eq:Gamma-recursion}. 
To complete the proof
by induction, we assume that it holds
for some $n\geq1$.  We use this assumed inequality at $n$
together with \eqref{eq:Gamma-recursion} at $n+1$.  The result is
\begin{align*}
\Gamma^{(n+1)}(t)&\leq 
c_0(1+\delta_\we^{-1}) m^{(n+1)}(t) + \we_\delta \,\Gamma^{(n)}(t)\\
& \le c_0(1+\delta_\we^{-1}) m^{(n+1)}(t) 
+ \we_\delta \Bigl(c_0(1+\delta_\we^{-1}) \ \sum_{k=0}^{n-1} \we_\delta^{k} \, m^{(n-k)}(t) 
+ 4 c_a^2 \we_\delta^n\Bigr)\\
&= c_0(1+\delta_\we^{-1}) \ \sum_{k=0}^{n} \we_\delta^k \, m^{(n+1-k)}(t) + 4c_a^2 \we_\delta^{n+1}.
\end{align*}

\noindent
\emph{Step 3.}  For $ t\in[0,T],\ n\geq1$,
we use  \eqref{eq:mn-recursion} and the recursion derived in Step 2 to estimate,
$$
m^{(n+1)}(t) \leq  
c^*(1-\we)\int_t^T \Gamma^{(n)}(s)\,\d s \leq 
c^* (1-\we)\left[ c_0 (1+\delta_\we^{-1})\sum_{k=0}^{n-1} \we_\delta^{k}\int_t^T m^{(n-k)}(s)\,\d s 
+ 4 c_a^2\we_\delta^n(T-t) \right].
$$
Hence, the following recursion holds for $m^{(n)}$
with the constant $\cld:= c^*(1-\we)\max \{ c_0 (1+\delta_\we^{-1}), 4c_a^2 \}$,
\begin{equation}
\label{eq:m-recursion}
    m^{(n+1)}(t) \leq \cld\ \left[\,\sum_{k=0}^{n-1} \we_\delta^k \int_t^T m^{(n-k)}(s)\,\d s 
    +  \we_\delta^n (T-t)\,\right],\qquad t\in[0,T],\ n\geq1.
\end{equation}
By \eqref{eq:mn-recursion},
$$
m^{(1)}(t) \le c^* (1-\we) \int_t^T \Gamma^{(0)}(s) \d s
\le  c^*(1-\we) 4c_a^2  (T-t)  \le \cld  (T-t).
$$

\noindent
\emph{Step 4.}  Set
$$
    M^{(n)}(t) := \we_\delta^n \sum_{k=1}^n\frac{1}{k!}\binom{n-1}{k-1}
    \left(\frac{\cld(T-t)}{\we_\delta}\right)^k,\quad t\in[0,T],\ n\geq1.
$$
We claim that  $m^{(n)} \leq M^{(n)}$ for every $n \ge 1$.
As  $M^{(1)}(t) = \cld (T-t)$,
and  the inequality $m^{(1)}(t) \leq  M^{(1)}(t)$
is proved above.
Assume the statement holds for all $k\leq n$ for some $n\geq1$. 
Then, for any $k\leq n$,
\begin{align*}
    \int_t^T m^{(n-k)}(s)\,\d s &\leq \int_t^T M^{(n-k)}(s)\,\d s
    = \we_\delta^{n-k} \sum_{\ell=1}^{n-k}\frac{1}{\ell!}\binom{n-k-1}{\ell-1}\int_t^T\left(\frac{\cld (T-s)}{\we_\delta}\right)^\ell\,\d s\\
    &=\we_\delta^{n-k} \sum_{\ell=1}^{n-k}\frac{1}{(\ell+1)!}\binom{n-k-1}{\ell-1}\left( \frac{\cld(T-t)}{\we_\delta}\right)^\ell\,(T-t).
\end{align*}
This leads to
\begin{align*}
    m^{(n+1)}(t) &\leq \cld \sum_{k=0}^{n-1}  \we_\delta^k\int_t^Tm^{(n-k)}(s)\,\d s 
    + \cld \we_\delta^n\,(T-t) \\
    &\leq \cld \sum_{k=0}^{n-1}  \we_\delta^k\left(\we_\delta^{n-k} \sum_{\ell=1}^{n-k}\frac{1}{(\ell+1)!}\binom{n-k-1}{\ell-1}
    \left( \frac{\cld (T-t)}{\we_\delta}\right)^\ell\right) \, (T-t)+ \cld \we_\delta^n\,(T-t)\\
    &=\we_\delta^{n+1}\sum_{\ell=1}^n \frac{1}{(\ell+1)!}\left(\frac{\cld (T-t)}{\we_\delta}\right)^{\ell+1}\  
    \sum_{k=0}^{n-\ell} \binom{n-k-1}{\ell-1} + \cld \we_\delta^n\,(T-t) \\
    &=\we_\delta^{n+1}\sum_{\ell=1}^n \frac{1}{(\ell+1)!}\left(\frac{\cld (T-t)}{\we_\delta}\right)^{\ell+1} \binom{n}{\ell} 
    + \cld  \we_\delta^n\,(T-t) \\
    &=\we_\delta^{n+1} \sum_{\ell=1}^{n+1}\frac{1}{\ell!}\binom{n}{\ell-1}
    \left(\frac{\cld (T-t)}{\we_\delta}\right)^\ell.
\end{align*}
In the above,
the first inequality follows from Step 3, the fourth line utilizes the identity 
$$
\sum_{k=0}^{n-\ell} \binom{n-k-1}{\ell-1} = \sum_{k=\ell-1}^{n-1}\binom{k}{\ell-1}=\binom{n}{\ell}
$$
and we make an index shift in the last line.
\vspace{5pt}

\noindent
\emph{Step 5.} We continue by showing that, for any $\we \in (0,1)$,
$(v^{(n)})_{n\geq1}$ is a Cauchy sequence in the complete Banach space
$\cS=\cC([0,T]\times\sX;\R)$ with norm given by \eqref{eq:snorm}.
We first observe that, since
$$
\binom{n-1}{k-1}\leq \binom{n}{k}\qquad \text{for}\ 1\leq k\leq n,
$$
$$
M^{(n)}(t) \leq \we_\delta^n\ L^{(n)}(-\cld (T-t)/\we_\delta),\quad \text{where}\quad 
L^{(n)}(x)=\sum_{k=0}^n \binom{n}{k}\frac{(-1)^k}{k!} x^k 
$$
is the $n$-th Laguerre polynomial.  Since the Laguerre polynomials 
are decreasing on $(-\infty,0]$,
$$
\|v^{(n+1)}-v^{(n)}\|_\cS =\sup_{t \in [0,T]}\big(m^{(n)}(t)\big)^{1/2} 
\le \bigl(M^{(n)}(0)\bigr)^{1/2}  \le
\we_\delta^{n/2}\ \big(L^{(n)}(-T_\delta )\bigr)^{1/2},
$$
where $T_\delta := \cld T/\we_\delta$.
Then, by the Cauchy-Schwarz inequality,
$$
\Big(\sum_{n\geq1}\|v^{(n+1)}-v^{(n)}\|_\cS\Big)^2
\le   \Big(\sum_{n\geq1}  \we_\delta^{n/2}\Big)\  \sum_{n\geq1}   \we_\delta^{n/2}\ L^{(n)}(-T_\delta)
=: c_\delta \ g( \we_\delta^{1/2}, -T_\delta),
$$
where $c_\delta:= \sum_{n\geq1}  \we_\delta^{n/2}$, 
and $g(x,t):= \sum_{n\geq1}   t^n L^{(n)}(x)$ is characteristic function
of the Laguerre polynomials.  As $\delta_\we<1$, $c_\delta<\infty$.
Also, it is known that 
$$
g(x,t)= (1-t)^{1/2}\ e^{- \tfrac{xt}{1-t}},
$$
and consequently, $g( \we_\delta^{1/2}, -T_\delta) <\infty$.
Hence, $\sum_{n\geq1} \|v^{(n+1)}-v^{(n)}\|_\cS <\infty$.
This implies that $(v^{(n)})_{n\geq1}$ is a Cauchy sequence in $\cS$. 
We complete the proof as in the Picard iteration case.
\end{proof}

\subsection{Convergence rates}%{Proof of Proposition \ref{prop:rate}}
\label{ss:rate}

\begin{proof}[Proof of Proposition \ref{prop:rate}]
Fix $\we\in(0,1)$.
We use the inequality $m^{(n)} \le M^{(n)}$ and the
definitions to obtain
$$
\|v^{(n)}-v\|_\cS \le \sum_{k=0}^\infty \|v^{(n+k+1)}-v^{(n+k)}\|_\cS
= \sum_{k=0}^\infty \big(m^{(n+k)}(0)\big)^{1/2} \le \sum_{k=0}^\infty \big(M^{(n+k)}(0)\big)^{1/2}.
$$
Recall that with   $T_\delta = \cld T/\we_\delta$,
$$
M^{(n)}(0) = \we_\delta^n \sum_{k=1}^n\frac{1}{k!}\binom{n-1}{k-1} T_\delta^k
$$
Since $\we_\delta <1$,
there is $k_\we$ such that $k! \ge (T_\delta)^k (1-\we_\delta)^{-k}$ for every $k \ge k_\we$. Hence,
\begin{align*}
M^{(n)}(0) & \le \we_\delta^n \ \left[
\sum_{k=1}^{k_\we} \frac{1}{k!}\binom{n-1}{k-1} T_\delta^k
+\sum_{k=k_\we+1}^n \binom{n-1}{k-1} (1-\we_\delta)^{-k}\right]\\
&\le \we_\delta^n \left[ 
\sum_{k=1}^{k_\we} \frac{1}{k!}\binom{n-1}{k-1} T_\delta^k
+\sum_{k=0}^n \binom{n}{k} (1-\we_\delta)^{-k}\right]
= \we_\delta^n \left[\hat{c}_\we+(2-\we_\delta)^n\right],
\end{align*}
where $\hat{c}_\we:=\sum_{k=1}^{k_\we} \frac{1}{k!}\binom{n-1}{k-1} T_\delta^k$.
Set $\gamma_\we :=  \we_\delta (2-\we_\delta)$, so that 
$\gamma_\we=1-(1-\we_\delta)^2<1$. As $\gamma_\we>\we_\delta$,
$$
M^{(n)}(0) \le \gamma_\we^n(\hat{c}_\we+1) ,
\qquad \Rightarrow \qquad
 \sum_{k=0}^\infty \big(M^{(n+k)}(0)\big)^{1/2} \le  \gamma_\we^n(\hat{c}_\we+1) 
 \sum_{k=0}^\infty \gamma_\we^k =: c_\we \gamma_\we^n.
$$

The case $\we=0$ is simpler
and follows essentially from the above proof, \emph{mutatis mutandis}.
\end{proof}

\section{Error analysis}
\label{sec:errors}
In the above convergence analysis, we assume that the optimal
response function $\alpha^*(\beta)$ is calculated perfectly, without any error. However, in practice $\alpha^*(\beta)$ is approximated using a numerical method, as we discuss in the next section. 
In this section, we investigate the robustness of the algorithm from the theoretical viewpoint by introducing an error in the best response calculation and investigate its propagation through the iterations.  Namely, 
we fix $0\le\rho<1$ and let $(v^{(n)},\alpha^{(n)},\beta^{(n)})_{n\geq1}$
be the sequence analyzed in subsection~\ref{ss:LPicard}. We use the  initialization 
$\beta^{(0)}= \hat \beta^{(0)} \in \sA^*$ 
and define the sequence $(\hat{v}^{(n)},\hat{\alpha}^{(n)},\hat{\beta}^{(n)})_{n\geq1}$ 
corrupted with error by
\begin{equation}\label{eq:error-def}
\hat{v}^{(n)}:=v^{\hat{\beta}^{(n-1)}}, \quad 
\hat{\alpha}^{(n)}:=\alpha^*(\hat{\beta}^{(n-1)}),\quad 
\hat{\beta}^{(n)}=\we \hat{\beta}^{(n-1)}+(1-\we)
\bigl(\hat{\alpha}^{(n)}+ \varepsilon^{(n)}\bigr),\quad n\ge 1,
\end{equation}
where $\varepsilon^{(n)}:[0,T]\times\sX\mapsto\R_+^{d-1}$ is the error introduced at the $n$-th iteration. 
Set
$$
q^{(n)}(t):=|\hat{v}^{(n)}(t)-v^{(n)}(t)|_2^2,\qquad 
\gamma^{(n)}(t):=|\hat{\beta}^{(n)}(t)-\beta^{(n)}(t)|^2_2,\qquad t\in[0,T], \ n\geq1,
$$
where $|\cdot|_2$ is an in \eqref{eq:norm}.
Recall the norm  $\|\cdot\|_\cS$ defined in \eqref{eq:snorm}, 
the constant $c^*$ of Lemma \ref{lem:value-leq-beta} and define
\begin{align*}
\Gamma^{(n-1)}(t) &:= \hcr\ \sum_{k=1}^{n-1}\|\varepsilon^{(n-k)}\|_{\cS}^2 \sum_{j=1}^k \binom{k-1}{j-1} \we^{k-j} \frac{c_{*,\rho}^{j-1}(T-t)^{j-1}}{(j-1)!}\ ,\\
Q^{(n)}(t) &:=  \sum_{k=1}^{n-1} \|\varepsilon^{(n-k)}\|_{\cS}^2\  \sum_{j=1}^k\binom{k-1}{j-1}
\we^{k-j}\ \frac{c_{*,\rho}^{j}(T-t)^{j}}{j!}\, ,
\end{align*}
where $c_{*,\rho}:= c^* \hcr$, $\hcr:=2(1-\we)^2 (c_0^2\vee 1)$ 
and $c_0$ is defined in Lemma \ref{lem:alpha-leq-value}.

\begin{Thm}
\label{th:cv-error}
For $\we \in (0,1)$, 
\begin{equation}\label{eq:error-estimate-claim-i}
\gamma^{(n-1)}(t)\leq \Gamma^{(n-1)}(t),\qquad q^{(n)}(t)\leq Q^{(n)}(t),
\qquad t\in[0,T],\ n\geq2.
\end{equation}
In particular, there exists a constant $\tilde c>0$ independent of $n$ such that 
\begin{equation}\label{eq:error-estimate-claim-ii}
\sup_{n\geq1} \ \big(\|\gamma^{(n)}\|_\infty+\|q^{(n)}\|_\infty\big)\leq \tilde c\  \sup_{n\geq1} \|\varepsilon^{(n)}\|_\cS^2.    
\end{equation}
\end{Thm}

The analysis of the pure Picard iteration can also be carried
out by following the proof below.  We provide the statement for $\we=0$
at the end of this section.  Additionally, the above error
estimate implies that when the errors are uniformly small,
then $\hat \beta^{(n)}$, for sufficiently large $n$, is an
approximate Nash equilibrium.  We state and prove this 
approximation result after the proof.

\begin{proof}
As $\hat{v}^{(n)}$ and $v^{(n)}$ solve the dynamic 
programming equation \eqref{eq:HJB-beta} with controls 
$\beta=\hat{\beta}^{(n-1)}$ and $\beta=\beta^{(n-1)}$, respectively, we see that
\begin{equation}\label{eq:error-q-leq-gamma}
q^{(n)}(t)\leq c^*\int_t^T |\hat{\beta}^{(n-1)}(s)-\beta^{(n-1)}(s)|^2_2 \, \d s=c^*\int_t^T \gamma^{(n-1)}(s)\, \d s,\quad n\geq 1,
\end{equation}
where $c^*$ is the constant from Lemma~\ref{lem:value-leq-beta}. By \eqref{eq:error-def}, for $t\in[0,T]$ and $n\geq1$,
$$
\hat{\beta}^{(n)}(t)-\beta^{(n)}(t)= \we(\hat{\beta}^{(n-1)}(t)-\beta^{(n-1)}(t))+(1-\we)(\hat{\alpha}^{(n)}(t)-\alpha^{(n)}(t))+(1-\we)\varepsilon^{(n)}(t).
$$
By Jensen's inequality, we may estimate
\begin{align*}
\gamma^{(n)}(t) &= \bigl|\we(\hat{\beta}^{(n-1)}(t)-\beta^{(n-1)}(t))+(1-\we)(\hat{\alpha}^{(n)}(t)-\alpha^{(n)}(t)+\varepsilon^{(n)}(t))\bigr|_2^2\\
&\leq \rho \gamma^{(n-1)}(t) + (1-\rho) |\hat{\alpha}^{(n)}(t)-\alpha^{(n)}(t)+\varepsilon^{(n)}(t)|_2^2\\
&\leq \rho \gamma^{(n-1)}(t) + 2(1-\rho)|\hat{\alpha}^{(n)}(t)-\alpha^{(n)}(t)|_2^2 + 2(1-\rho)|\varepsilon^{(n)}(t)|_2^2,
% &\leq  \rho \gamma^{(n-1)}(t) + 2(1-\rho)c_0^2 q^{(n)}(t) + 2(1-\rho)|\varepsilon^{(n)}(t)|_2^2\\
% &\leq \rho\gamma^{(n-1)}(t) + \hcr (q^{(n)}(t) + |\varepsilon^{(n)}(t)|_2^2),
\end{align*}
where we used the inequality
$$
|\hat\alpha^{(n)}(t)-\alpha^{(n)}(t)+\varepsilon^{(n)}(t)|^2_2 \leq 2|\hat\alpha^{(n)}(t)-\alpha^{(n)}(t)|_2^2+2|\varepsilon^{(n)}(t)|_2^2.
$$
By Lemma~\ref{lem:alpha-leq-value},
$\alpha^{(n)}=\alpha^*(\beta^{(n-1)})$,
and the definitions  $\hat{v}^{(n)}=v^{\hat{\beta}^{(n-1)}}$, $v^{(n)}=v^{\beta^{(n-1)}}$,
$$
|\hat\alpha^{(n)}(t)-\alpha^{(n)}(t)|_2^2 =|\alpha^*(\hat\beta^{(n-1)})(t)-\alpha^*(\beta^{(n-1)})(t)|_2^2 
\le c_0^2 | \hat{v}^{(n)}(t)- v^{(n)}(t)|_2^2 
= c_0^2 q^{(n)}(t).
$$
Hence,
\begin{equation}\label{eq:error-gamma-weighted}
\gamma^{(n)}(t)\leq \rho\gamma^{(n-1)}(t) + \hcr (q^{(n)}(t) + |\varepsilon^{(n)}(t)|_2^2)\quad \text{with }\quad \hat c_{\rho} := 2(1-\rho)(c_0^2\lor 1).
\end{equation}

We now prove the estimate \eqref{eq:error-estimate-claim-i} by induction.
\vspace{5pt}

\noindent \emph{Base case.} 
Since $q^{(1)}\equiv\gamma^{(0)}\equiv0$,
\eqref{eq:error-gamma-weighted} with $n=2$ implies that
$$
\gamma^{(1)}(t)\leq \hcr \|\varepsilon^{(1)}\|_\cS^2 = \Gamma^{(1)}(t),
$$
and hence, by \eqref{eq:error-q-leq-gamma} and the definition $c_{*,\rho}=c^* \hcr$,
$$
q^{(2)}(t)\leq c^*\int_t^T \gamma^{(1)}(s)\,\d s \leq c_{*,\rho} \|\varepsilon^{(1)}\|_\cS^2 (T-t) = Q^{(2)}(t).
$$

\noindent \emph{Induction step.} Assume the statement \eqref{eq:error-estimate-claim-i} holds for some $n\geq2$. By \eqref{eq:error-gamma-weighted} and the induction assumption,
\begin{equation}\label{eq:error-gamma-weighted-proof}
\gamma^{(n)}(t)\leq \we\gamma^{(n-1)}(t)+\hcr(q^{(n)}(t)+|\varepsilon^{(n)}(t)|_2^2)\leq \we\Gamma^{(n-1)}(t)+\hcr(Q^{(n)}(t)+\|\varepsilon^{(n)}\|_\cS^2)\leq \cI+\cJ+\cK,
\end{equation}
where 
\begin{align*}
\cI&:= \hcr \sum_{k=1}^{n-1}\|\varepsilon^{(n-k)}\|_{\cS}^2 \sum_{j=1}^k \binom{k-1}{j-1}\we^{k+1-j}\ \frac{c_{*,\rho}^{j-1}(T-t)^{j-1}}{(j-1)!}, \\
\cJ &:= \hcr \sum_{k=1}^{n-1} \|\varepsilon^{(n-k)}\|_{\cS}^2\  \sum_{j=1}^k\binom{k-1}{j-1}
\we^{k-j}\ \frac{c_{*,\rho}^{j}(T-t)^{j}}{j!},\qquad \cK:=\hcr\|\varepsilon^{(n)}\|_{\cS}^2.
\end{align*}
We rewrite $\cI$ and $\cJ$ as follows:
\begin{align*}
\cI&=\hcr\sum_{k=1}^{n-1}\|\varepsilon^{(n-k)}\|_{\cS}^2 \left(\rho^k +\sum_{j=2}^k \binom{k-1}{j-1}\ \rho^{k+1-j} \frac{c_{*,\rho}^{j-1}(T-t)^{j-1}}{(j-1)!}\right)\\
&=\hcr\sum_{k=1}^{n-1}\|\varepsilon^{(n-k)}\|_{\cS}^2\left(\rho^k+\sum_{j=1}^{k-1} \binom{k-1}{j} \rho^{k-j}\frac{c_{*,\rho}^{j}(T-t)^{j}}{j!}\right),\\
\cJ &=\hcr\sum_{k=1}^{n-1} \|\varepsilon^{(n-k)}\|_{\cS}^2\left(\frac{c_{*,\rho}^{k}(T-t)^{k}}{k!}  + \sum_{j=1}^{k-1}\binom{k-1}{j-1}\
\rho^{k-j}\ \frac{c_{*,\rho}^{j}(T-t)^{j}}{j!}\right).
\end{align*}
Then, summing $\cI$ and $\cJ$ leads to
\begin{align*}
\cI+\cJ&=\hcr\sum_{k=1}^{n-1}\|\varepsilon^{(n-k)}\|_{\cS}^2\left(\rho^k +\frac{c_{*,\rho}^{k}(T-t)^{k}}{k!}+\sum_{j=1}^{k-1} \left(\!\binom{k-1}{j}+\binom{k-1}{j-1}\!\right) \rho^{k-j}\frac{c_{*,\rho}^{j}(T-t)^{j}}{j!}\right)\\
&=\hcr\sum_{k=1}^{n-1}\|\varepsilon^{(n-k)}\|_{\cS}^2\left(\rho^k +\frac{c_{*,\rho}^{k}(T-t)^{k}}{k!}+\sum_{j=1}^{k-1} \binom{k}{j} \rho^{k-j}\frac{c_{*,\rho}^{j}(T-t)^{j}}{j!}\right)\\
&=\hcr\sum_{k=1}^{n-1}\|\varepsilon^{(n-k)}\|_{\cS}^2\sum_{j=0}^k \binom{k}{j} \rho^{k-j}\frac{c_{*,\rho}^{j}(T-t)^{j}}{j!} \\
&= \sum_{k=2}^n \|\varepsilon^{(n+1-k)}\|_\cS^2\sum_{j=1}^k \binom{k-1}{j-1}\rho^{k-j} \frac{c_{*,\rho}^{j-1}(T-t)^{j-1}}{(j-1)!}.
\end{align*}
Here, we used the identity
$$
\binom{k-1}{j}+\binom{k-1}{j-1} = \binom{k}{j}
$$
and performed two index shifts in the last line.
Combining this with \eqref{eq:error-gamma-weighted-proof} shows that 
$$
\gamma^{(n)}(t) \leq \cI+\cJ+\cK = \hcr \sum_{k=1}^n\|\varepsilon^{(n+1-k)}\|_\cS^2\sum_{j=1}^k \binom{k-1}{j-1}\rho^{k-j} \frac{c_{*,\rho}^{j-1}(T-t)^{j-1}}{(j-1)!}=\Gamma^{(n)}(t).
$$

To complete the induction step, we now show that $q^{(n+1)}(t)\leq Q^{(n+1)}(t)$. Combining \eqref{eq:error-q-leq-gamma} and the just established inequality $\gamma^{(n)}(t) \le \Gamma^{(n)}(t)$,
\begin{align*}
q^{(n+1)}(t)&\leq c^*\int_t^T \gamma^{(n)}(s)\d s \leq
c^* \int_t^T \Gamma^{(n)}(s)\d s\\
&= c^* \int_t^T \hcr \sum_{k=1}^n\|\varepsilon^{(n+1-k)}\|_\cS^2\sum_{j=1}^k \binom{k-1}{j-1}\rho^{k-j} \frac{c_{*,\rho}^{j-1}(T-s)^{j-1}}{(j-1)!}\ \d s\\
&= c^* \hcr \sum_{k=1}^n\|\varepsilon^{(n+1-k)}\|_\cS^2\sum_{j=1}^k \binom{k-1}{j-1}\rho^{k-j} \frac{c_{*,\rho}^{j-1}(T-t)^{j}}{j!} = Q^{(n+1)}(t).
\end{align*}
This completes the induction step. 
\vspace{5pt}

Finally, let $\varepsilon_*:=\sup_{n\geq1} \|\varepsilon^{(n)}\|_\cS^2$. 
The proof of Theorem~\ref{thm:convergence} implies that
$$
\tilde c := (1+\hcr/\rho)\sum_{k=1}^\infty\rho^k L^{(k)}(-\csr T/\rho)<\infty.
$$
This and the definitions  of  $\Gamma^{(n)}(t)$ and $Q^{(n)}(t)$ imply that
$\Gamma^{(n)}(t) + Q^{(n)}(t) \leq \varepsilon_* \tilde c$, so the estimate \eqref{eq:error-estimate-claim-ii} follows.
\end{proof}

We next show that the $\beta^{(n)}$ form approximate Nash equilibria,
as defined below.

\begin{Def}
\label{def:epsilon_Nash}
{\rm{For $\varepsilon>0$, we say that $\gamma^* \in \sA$ is
an}} $\varepsilon$-Markov perfect equilibrium, {\rm{if
$$
J(t,x,\mu,\gamma^*;\gamma^*)
\leq J(t,x,\mu,\gamma;\gamma^*) + \varepsilon,
\qquad \forall \, (t,x,\mu,\gamma)\in[0,T]\times\sX\times\sA,
$$
where $J$ is the pay-off functional given by \eqref{eq:payoff}.}}
\end{Def}
We first prove a general result.  Recall the value function $v^\beta$
of the tagged player when the other players use the feedback strategy $\beta\in\sA$.

\begin{Lem}
\label{lem:approximate_Nash} 
Let $\alpha^*$ be the unique Markov perfect equilibrium. Then, there exists $\kappa_*>0$ such that, 
if $\gamma^*\in\sA$ satisfies
$$
\|\gamma^*-\alpha^*\|_\cS + \|v^{\gamma^*}-v^{\alpha^*}\|_\cS\leq \delta,
$$
for some  $\delta  \in (0,1]$, then  
$\gamma^*$ is an $\kappa_* \delta$-Markov perfect equilibrium.
\end{Lem}

\begin{proof}
We fix  $\delta \in (0,1]$ and $\gamma^* \in \sA$ satisfying the hypothesis.
Let $\alpha^*(\gamma^*)$ be the best response of the tagged player,
and $\tilde\sA$ be the set of all admissible controls whose $\|\cdot\|_\cS$ norm is
bounded by $c_a+1$. 
Note that $\sA^*\subset\tilde \sA$, where $\sA^*$ is defined in \eqref{eq:smallA}.
Further, in view of Lemma \ref{lem:constants}, $\alpha^*, \alpha^*(\gamma^*) \in \sA^*$,
and therefore $\gamma^* \in \tilde \sA$.  
As $\ell$ is convex, there exists a constant $c_1>0$ such that
$$
|\ell(x,\mu,a) - \ell(x,\mu,a')| 
\le c_1 |a-a' |_2,
$$
for any $(x,\mu)\in\sX$ and $a,a' \in \R_+^{d-1}$ satisfying $|a|_2,|a'|_2\leq c_a+1$. We proceed in several steps.\\

\noindent
\emph{Step 1.} Let $\alpha^*(\gamma^*)$ be the best response of the tagged player. 
By Lemma \ref{lem:control-lip},
\begin{align*}
|\alpha^*(\gamma^*)(t,x,\mu)-\alpha^*(t,x,\mu)|_2 
&= |\hat \alpha(x,\mu,\Delta_x v^{\gamma^*}(t,\cdot,\mu))-\hat \alpha(x,\mu,\Delta_x v^{\alpha^*}(t,\cdot,\mu))|_2 \\
&\leq c_a^* |\Delta_x v^{\gamma^*}(t,\cdot,\mu)-\Delta_x v^{\alpha^*}(t,\cdot,\mu)|_2,
\qquad \forall (t,x,\mu)\in[0,T]\times\sX.
\end{align*}
Hence,  $\|\alpha^*(\gamma^*)-\alpha^*\|_\cS \leq 4c_a^*(d-1)\delta.$ 
By the triangle inequality, there exists $c_2>0$ such that 
$$
\|\gamma^*-\alpha^*(\gamma^*)\|_\cS\leq c_2\delta.
$$

\noindent
\emph{Step 2.} For $(t,x,\mu)\in[0,T]\times\sX$, set 
$$
j(t,x,\mu):=J(t,x,\mu,\alpha(\gamma^*);\gamma^*)=v^{\gamma^*}(t,x,\mu),\qquad k(t,x,\mu):=J(t,x,\mu,\gamma^*;\gamma^*),
$$
and
$$
w(t,x,\mu):= j(t,x,\mu)-k(t,x,\mu),\qquad m(t):=|w(t)|_2^2.
$$
By the Feynman-Kac formula, the functions $j$ and $k$ satisfy, for $(t,x,\mu)\in[0,T]\times\sX$,
\begin{align*}
-\frac{\d}{\d t} j(t,x,\mu) &= \ell(x,\mu,\alpha^*(\gamma^*)(t,x,\mu)) + \sL_x^{\alpha^*(\gamma^*)}j(t)(t,x,\mu)+\sL^{\gamma^*}_\mu j(t)(t,x,\mu),\\
-\frac{\d}{\d t} k(t,x,\mu) &= \ell(x,\mu,\gamma^*(t,x,\mu)) + \sL_x^{\gamma^*}k(t)(t,x,\mu)+\sL^{\gamma^*}_\mu k(t)(t,x,\mu),
\end{align*}
with the common terminal condition
$$
j(T,x,\mu)=g(x,\mu)=k(T,x,\mu),\qquad (x,\mu)\in\sX.
$$
For $t\in[0,T]$, we compute
$$
-\frac{\d}{\d t} m(t) = - \sum_{(x,\mu)\in\sX} 
 \frac{\d}{\d t}|w(t,x,\mu)|_2^2 =  \sum_{(x,\mu)\in\sX} \left(-\frac{\d}{\d t} w(t,x,\mu)\right) w(t,x,\mu) =\cI+\cJ+\cK,
$$
where, for $(t,x,\mu)\in[0,T]\times\sX$,
\begin{align*}
\cI &:= \sum_{(x,\mu)\in\sX}\bigl[\ell(x,\mu,\alpha^*(\gamma^*)(t,x,\mu))-\ell(x,\mu,\gamma^*(t,x,\mu))\bigr] \, w(t,x,\mu),\\
\cJ &:= \sum_{(x,\mu)\in\sX} \bigl[\sL_x^{\alpha^*(\gamma^*)}j(t)(t,x,\mu)-\sL_x^{\gamma^*}k(t)(t,x,\mu)\bigr] \, w(t,x,\mu),\\
\cK &:= \sum_{(x,\mu)\in\sX} \sL^{\gamma^*}_\mu w(t)(t,x,\mu)\, w(t,x,\mu).
\end{align*}
For $\cI$, we use the Cauchy-Schwarz inequality, Step 1 and the fact that $m$ is bounded by a constant, depending on the pointwise bound $c_a+1$ of the controls, to see that there exists a constant $c_3>0$ such that
\begin{align*}
|\cI| &\leq c_1 \sum_{(x,\mu)\in\sX} | \alpha^*(\gamma^*)(t,x,\mu)- \gamma^*(t,x,\mu)|_2 |w(t,x,\mu)|\\
&\leq c_1 \|\alpha^*(\gamma^*)-\gamma^*\|_\cS \sqrt{m(t)} \leq c_1c_2\delta \sqrt{m(t)} \leq c_3\delta.
\end{align*}
For $\cJ$, using similar arguments we see that we can redefine $c_3$ such that 
\begin{align*}
|\cJ| &\leq \sum_{(x,\mu)\in\sX}   \bigl[|(\sL_x^{\alpha^*(\gamma^*)}-\sL_x^{\gamma^*})j(t)(t,x,\mu)| +|\sL_x^{\gamma^*}\big(j(t)-k(t)\big)(t,x,\mu)|\bigr] \, |w(t,x,\mu)|\\ 
&\leq \|\lambda^1\|_\infty\sum_{(x,\mu)\in\sX} \sum_{y\neq x} |\alpha^*(\gamma^*)_y(t,x,\mu)-\gamma^*_y(t,x,\mu)| \,  |j(t,y,\mu)-j(t,x,\mu)| |w(t,x,\mu)|\\
&+ \|\lambda^1\|_\infty\sum_{(x,\mu)\in\sX} \sum_{y\neq x}  \gamma^*_y(t,x,\mu) |w(t,y,\mu)-w(t,x,\mu)|\,|w(t,x,\mu)|\leq c_3(\delta + m(t)).
\end{align*}
Finally, for $\cK$, we again use the pointwise bound $c_a+1$ and redefine $c_3$ to see that 
\begin{align*}
|\cK| &\leq \sum_{(x,\mu)\in\sX} \sum_{y\neq z} n^\mu_z \left(\|\lambda^0\|_\infty+\|\lambda^1\|_\infty
\gamma^*_y(t,z,\mu+e_{x,z})\right)[w(t,x,\mu+e_{y,z})-w(t,x,\mu)] \, |w(t,x,\mu)|\\
&\leq c_3 m(t).
\end{align*}
Putting these estimates together yields the differential inequality
$$
-\frac{\d}{\d t} m(t) \leq 2c_3(\delta+m(t)).
$$
By Gr\"{o}nwall's inequality and the fact that $m(T)=0$, we obtain 
$$
0\leq m(t) \leq \delta(e^{2c_3(T-t)}-1),\quad $$
and hence
$$
    \|m\|_\infty\leq\delta(e^{2c_3T}-1) .
$$
This shows that there exist $\kappa_*>0$ with 
$$
\max_{(t,x,\mu)\in[0,T]\times\sX}|v^{\gamma^*}(t,x,\mu)-J(t,x,\mu,\gamma^*;\gamma^*)| =\max_{(t,x,\mu)\in[0,T]\times\sX} |w(t,x,\mu) | \leq \kappa_* \delta.
$$

\noindent
\emph{Step 3.} For $(t,x,\mu,\gamma)\in[0,T]\times\sX\times\sA$, we conclude that
$$
J(t,x,\mu,\gamma^*;\gamma^*) \leq v^{\gamma^*}(t,x,\mu) + \kappa_*  \delta 
\leq  J(t,x,\mu,\gamma;\gamma^*) + \kappa_* \delta.
$$
Hence, $\gamma^*$ is an $\epsilon$-Markov perfect equilibrium with $\epsilon = \kappa_* \delta$.
\end{proof}

Recall the constants $\tilde{c}$ defined in
Theorem \ref{th:cv-error}, $\kappa_*$ defined in Lemma \ref{lem:approximate_Nash},
and set
$$
\varepsilon_*:= \sup_n \|\varepsilon^{(n)}\|_\cS.
$$ 

\begin{Prop}
\label{pro:eN}
For any $\varepsilon_* \le 1/(\tilde{c}+1)$, there exists $n_*=n_*(\varepsilon_*)\ge1$
such that the strategies $\hat{\beta}^{(n)}$ are $ \kappa_* (\tilde{c}+1) \varepsilon_*$-Markov prefect equilibria
for every $ n \ge n_*$. 
\end{Prop}
\begin{proof}
Let $\alpha^*\in\sA^*$ be the unique Markov perfect equilibrium.
By Theorem \ref{thm:convergence}, there is $n_*\geq1$ such that
$\|\beta^{(n)}-\alpha^*\|_\cS 
+ \|v^{\beta^{(n)}}-v^{\alpha^*}\|_\cS\leq  \varepsilon_*$ for all $n \ge n_*$.
In view of Theorem \ref{th:cv-error}, for every $n \ge 1$, we have 
$\|\hat{\beta}^{(n)}-\beta^{(n)}\|_\cS 
+ \|v^{\hat{\beta}^{(n)}}-v^{\beta^{(n)}}\|_\cS\leq \tilde{c} \varepsilon_*$. Then,
by triangle inequality,
$$
\|\hat{\beta}^{(n)}-\alpha^*\|_\cS 
+ \|v^{\hat{\beta}^{(n)}}-v^{\alpha^*}\|_\cS\leq (\tilde{c}+1) \varepsilon_*,
\qquad \forall n \ge n_*.
$$
As by hypothesis $(\tilde{c}+1) \varepsilon_* \le1$,
we complete the proof by invoking Lemma \ref{lem:approximate_Nash}. 
\end{proof}

Clearly, with some direct analysis one can obtain the analogue
of the above result with no conditions on the error $\varepsilon_*$.
However, we chose to omit it as the result is inconsequential.
Finally, we state the analogous results for $\we=0$
without a proof.

\begin{Prop}
Set $\we=0$, let $(v^{(n)},\alpha^{(n)},\beta^{(n)})_{n\ge1}$ be as in subsection \ref{ss:Picard}, and let $(\hat v^{(n)},\hat \alpha^{(n)},\hat \beta^{(n)})_{n\ge1}$ and $(q^{(n)},\gamma^{(n)})_{n\ge1}$ be defined as above in \eqref{eq:error-def} but with $\rho=0$. Redefine
$$
\Gamma^{(n-1)}(t)= \hat c_0\sum_{k=1}^{n-1}\frac{c_{*,0}^{k-1}(T-t)^{k-1}}{(k-1)!}\|\varepsilon^{(n-k)}\|_\cS^2,\quad Q^{(n)}(t)= \sum_{k=1}^{n-1}\frac{c_{*,0}^k(T-t)^{k}}{k!}\|\varepsilon^{(n-k)}\|_{\cS}^2,\quad t\in[0,T],\ n\ge2,
$$
where $c_{*,0}=c^*\hat c_0$, $\hat c_0=2(c_0^2\lor1)$. Here, $c_0$ is defined in Lemma \ref{lem:alpha-leq-value} and $c^*$ is defined in Lemma \ref{lem:value-leq-beta}. Then,
$$
\gamma^{(n-1)}(t)\leq \Gamma^{(n-1)}(t),\qquad q^{(n)}(t)\leq Q^{(n)}(t),
\qquad t\in[0,T],\ n\geq2.
$$
Furthermore, the conclusion of Proposition \ref{pro:eN} holds with $\rho=0$.
\end{Prop}

\begin{Rmk}\label{rmk:propopgation}
{\rm{The analysis of this section shows that the error
propagation remains bounded throughout the iterations.
However, in many steps the estimates are not tight
and  that one might obtain better bounds.
In particular, we believe that 
under suitable conditions, %when the error processes are independent of each other,
in the weighted Picard iterations averaging takes place
making the variance of the error decrease rapidly.
As these are not the main goals
of the paper, we chose to relegate these intriguing
questions  to future studies.}}
\end{Rmk}

\section{Numerical algorithms}
\label{sec:numerical}

We propose two numerical schemes associated with the (weighted) Picard algorithms. As described in Section~\ref{sec:iterations}, at the core of any such algorithm lies the computation of the tagged player's best response $\alpha^*(\beta)$ given an action $\beta\in\sA$ of other players,
and we propose two schemes
to approximate $\alpha^*(\beta)$:
\begin{enumerate}[(i)]
\item an ODE solver for the the dynamic programming equation \eqref{eq:HJB-beta};
\item a Monte Carlo simulation with empirical risk minimization.
\end{enumerate}

\subsection{Iterations based on the HJB equation}

This subsection presents an algorithm that uses readily available packages to numerically compute solutions to ODEs. The crucial difference to the algorithm described in Section~\ref{ssec:neural-network} is that, in theory, this algorithm computes the MPE and does not depend on the initial distribution.

\begin{algorithm}[H]
\caption{(Weighted) Picard iteration based on ODE solvers}
\begin{algorithmic}[1]
\State \textbf{Input:} Time grid $\cT$, number of iterations $N_{\mathrm{iter}}$, number of players $N$, weighting constant $\we\in[0,1)$.
\State Initialize $\hat \beta^{(0)}:\cT\times\sX\mapsto\R_+^{d-1}$.
\For{$n = 1$ to $N_{\mathrm{iter}}$}
    \State Using an ODE solver, compute the solution $\hat v$ to \eqref{eq:HJB-beta} with $\beta=\hat\beta^{(n-1)}$ on the grid $\cT$. This yields a vector $\hat v:\cT\times\sX\mapsto\R$.
    \State Numerically compute the optimal control $\hat \alpha^{(n)}:\cT\times\sX\mapsto\R_+^{d-1}$ using the relation \eqref{eq:alpha*}.
    \State Set $\hat \beta^{(n)} \gets \we \hat  \beta^{(n-1)}+(1-\we)\hat \alpha^{(n)}$.
\EndFor
\State \textbf{Return} Final control $\beta^{(N_{\mathrm{iter}})}$.
\end{algorithmic}
\end{algorithm}

\begin{Rmk}
\label{rem:ode}
{\rm{We could also use readily available packages to solve the~\eqref{eq:N-NLL} equation directly.
In contrast, the above algorithm solves 
the dynamic programming equation \eqref{eq:HJB-beta}
in each Picard iteration. 
While possibly slower than the direct method,
this approach might be more stable.
Indeed, the time discretization of the~\eqref{eq:N-NLL} equation  
requires the time steps
to be sufficiently small to ensure unique of solutions,
as shown in~\cite{HSY}.  
In contrast,  \eqref{eq:HJB-beta} is a  dynamic programming equation  and, as such, 
it is always monotone and  does not 
suffer from this difficulty.  Although most ODE solver
employ implicit schemes and may not require the smallness of time steps,
we believe that this structural difference could become significant for 
large systems.  Additionally, in Section~\ref{sec:example-twostates}
we have numerically  observed that, for large values of $N$,
our method provides accurate results while the 
solution provided by the direct method does not yield a MPE.}}
\end{Rmk}

\subsection{Monte Carlo simulations with empirical risk minimization}\label{ssec:neural-network}

In this case, controls are directly represented by neural networks. As controls are functions on $[0,T]\times\cX\times\Sigma_N^{d-1}\subset\R\times\R^d\times\R^{d-1}$, the input dimension of a neural network representing these functions is $1+d+(d-1)=2d$. If there is a natural embedding of $\cX$ into $\R$, this dimension can be reduced to $d+1$. The output of the neural network are $d-1$ many jump rates. Given $\rho\in[0,1)$, we proceed as follows to approximate a MPE.

\begin{algorithm}[H]
\caption{Neural (weighted) Picard iteration\label{algo:nn-picard}}
\begin{algorithmic}[1]
\State \textbf{Input:} Time horizon $T$, number of iterations $N_{\mathrm{iter}}$, number of players $N$, weighting constant $\we\in[0,1)$, initial distribution $\theta_0\in\Sigma^d_{N+1}$ of tagged and untagged players, number of simulated trajectories $M$.
\State Initialize $\hat\beta^{(0)}$ as a neural network.
\For{$n = 1$ to $N_{\mathrm{iter}}$}
    \State Initialize $\hat\alpha^{(n)}$ as a neural network.
    \For{each training epoch}
        \State Monte Carlo simulation: Simulate $M$ trajectories by the pooled jump simulation described in the text with initial distribution $\theta_0$, in which the tagged player uses $\hat\alpha^{(n)}$ and other players use $\hat\beta^{(n-1)}$.
        \State Estimate the cost $\int_\sX J(x_0,\mu_0, \hat\alpha^{(n)}; \hat\beta^{(n-1)}) \,\theta_0(\d x_0\times \d \mu_0)$ using the simulated trajectories.
        \State Update the parameters of $\alpha^{(n)}$ via gradient descent.
    \EndFor
    \State Set $\hat\beta^{(n)} \gets \we \hat\beta^{(n-1)}+(1-\we)\hat\alpha^{(n)}$.
\EndFor
\State \textbf{Return} Final control $\hat\beta^{(N_{\mathrm{iter}})}$.
\end{algorithmic}
\end{algorithm}

The estimation of the expected cost $J$ is done by Monte Carlo simulations of the trajectories of both the tagged and untagged players, starting from a prescribed initial distribution. At time $t\in[0,T]$ we recall that that $X_t\in\cX$ denotes the position of the tagged player while $\mu_t \in \Sigma^{d-1}_N$ tracks the empirical distribution of the untagged players. From an untagged player's view, who finds themselves in a state $z\in\cX$, the empirical distribution of others is given by 
$$
\mu_t  +e_{X_t,z} = \mu_t + \frac1N (e_{X_t}-e_z).
$$
We now assume that all other players use a control $\beta\in\sA$ while the tagged player uses $\alpha\in\sA$ and we wish to simulate trajectories for $\Xi_t = (X_t,\mu_t)$ according to the measure $\P^{\alpha\otimes\beta}$. Given a time $t\in[0,T]$ and a current state $\Xi_t=(x,\mu)\in\sX$, we proceed as follows. 
\begin{enumerate}
\item \textbf{Pooled jump simulation.} Compute the total jump rate of both the tagged and untagged players:
$$
\Lambda(t,x,\mu) := \Lambda^{\mathrm{tagged}}(t,x,\mu) + \Lambda^{\mathrm{untagged}}(t,x,\mu),
$$
where 
\begin{align*}
\Lambda^{\mathrm{tagged}}(t,x,\mu) &:= \sum_{y\neq x} \bigl( \lambda^0_y(x,\mu)+\lambda^1_y(x,\mu)\alpha_y(t,x,\mu)\bigr ),\\
\Lambda^{\mathrm{untagged}}(t,x,\mu) &:= \sum_{z\in\cX}\Lambda_z^{\mathrm{untagged}}(t,x,\mu),\\
\Lambda_z^{\mathrm{untagged}}(t,x,\mu) &:= \sum_{y\neq z}n^\mu_z \, \bigl(\lambda^0_y(z,\mu+e_{x,z})+\lambda^1_y(z,\mu+e_{x,z})\beta_y(t,z,\mu+e_{x,z})\bigr),\quad z\in\cX.
\end{align*}
Then:
\begin{itemize}
\item Sample an exponential waiting time $\tau \sim \mathrm{Exponential}(\Lambda(t,x,\mu))$.
\item Randomly choose the tagged or an untagged player, identified by their current state, according to the distribution 
$$
\left(\frac{\Lambda^{\mathrm{tagged}}(t,x,\mu)}{\Lambda(t,x,\mu)},\frac{\Lambda_1^{\mathrm{untagged}}(t,x,\mu)}{\Lambda(t,x,\mu)},\ldots,\frac{\Lambda_d^{\mathrm{untagged}}(t,x,\mu)}{\Lambda(t,x,\mu)}\right)\in[0,1]^{d+1}.
$$
\item If the tagged player is chosen, sample a new state according to the controlled jump rates, i.e., if $X'\in\cX$ denotes the new state, then
$$
\P(X'=y) = \frac{\lambda^0_y(x,\mu)+\lambda^1_y(x,\mu)\alpha_y(t,x,\mu)}{\Lambda^{\mathrm{tagged}}(t,x,\mu)},\quad y\neq x.
$$
Similarly, if an untagged player in a state $z\in\cX$ is chosen, and $Z'$ denotes the new state of this player, then
$$
\P(Z'=y) = \frac{\lambda^0_y(z,\mu+e_{x,z})+\lambda^1_y(z,\mu+e_{x,z})\beta_y(t,z,\mu+e_{x,z})}{\Lambda^{\mathrm{untagged}}_z(t,x,\mu)},\quad y\neq z.
$$
\end{itemize}

\item \textbf{Update:} The system time is updated to $t+\tau$. If the tagged player jumps, then we update 
$$
\Xi_{t+\tau} = (X_{t+\tau},\mu_{t+\tau}) = (X',\mu_t).
$$
If an untagged player in state $z\in\cX$ jumps, then we set 
$$
\Xi_{t+\tau} = (X_{t+\tau},\mu_{t+\tau}) = (X_t, \mu_t + e_{Z',z}).
$$
This process is repeated until terminal time $T$.
\end{enumerate}

This pooled jump scheme allows exact simulation of continuous-time dynamics without time discretization.
The control networks $\alpha$ and $\beta$ are implemented as fully connected neural networks with two or more hidden layers 
%\fh{should we specify the number of layers in this general algorithm?} 
and  softplus activation on the output layer to ensure positivity of jump rates.
To estimate the expected cost $J$, we simulate multiple independent trajectories using the pooled jump simulator described above. 
The optimal control is estimated using empirical risk minimization. Gradients are backpropagated 
through the neural network using the Adam algorithm.

\section{Examples}
\label{sec:examples}
In this section, we apply  the two algorithms outlined in the previous section
to the two state Kuramoto synchronization games and to a four-state
cyber-security model.

\subsection{Two-state examples without uniqueness in the mean-field}
\label{sec:example-twostates}

In this section, we first consider two MFG problems. In the corresponding finite-player problems, we then compare Picard iterations based on the HJB equation with a direct approach of solving the \eqref{eq:N-NLL} equation. The two-state problems we consider exhibit synchronization phenomena that result in non-unique solutions in the mean-field limit, with solutions to the mean-field NLL equation (or master equation) becoming discontinuous for sufficiently large time horizons. By interpreting the mean-field NLL equation as a scalar conservation law, we show that the unique entropy solution to the mean-field NLL can be approximated by our Picard algorithms for the finite-player game. Additionally, we demonstrate the instability of a standard ODE solver when directly solving \eqref{eq:N-NLL} in certain parameter regimes. We consider two variants of the problem. 

\emph{First MFG model.} Our first model follows~\cite{cecchin2019convergence}, which considers, for a \emph{coupling strength} $\kappa>0$,  
$$
\cX=\{0,1\},\quad \lambda^0\equiv 0,\quad \lambda^1\equiv1,\quad \ell(x,\mu,a)=\frac{1}{2} a^2 ,\quad g(x,\mu) \equiv \kappa\big(\mu(\{1\})\chi_{x=0}+\mu(\{0\})\chi_{x=1}\big)
$$
for $(x,\mu,a)\in\cX\times\Sigma^1\times\R_+$. Note that the terminal coupling cost encourages synchronization and satisfies
$$
\int_\cX (g(x,\mu')-g(x,\mu))\, (\mu'-\mu)(\d x) = 2\kappa(\mu'-\mu)(\{1\})(g'-g)(\{0\})  = - 2\kappa (\mu'-\mu)(\{0\})^2\leq0,
$$
for any $\mu,\mu'\in\Sigma^1$, and, hence, is not Lasry-Lions monotone. In this two-state model, we identify probability measures $\mu=(\mu(\{0\}),\mu(\{1\}))\in\Sigma^1$ with their value $p = \mu(\{1\})\in[0,1]$. Then, the mean-field NLL equation reads, for $(t,x,p) \in [0,T] \times \cX \times [0,1]$, 
\begin{align*}
-\partial_t \mathcal{U}(t,x,p) &= -
\frac{1}{2} \left[ \left( \cU(t,1-x,p)-\cU(t,x,p) \right)^- \right]^2 -\partial_p \cU(t,x,p)\,[\cU(t,0,p)-\cU(t,1,p)]^-\,p\\
&\quad +\, \partial_p \cU(t,x,p)\,[\cU(t,0,p)-\cU(t,1,p)]^+\,(1-p),
\end{align*}
with the terminal condition $\cU(T,x,p)=g(x,p)$. This system can be recognized as a scalar conservation law by introducing
$$
\mathcal{V}(t,q) = \mathcal{U}(T - t, 0,(q+1)/2) - \mathcal{U}(T - t, 1, (q+1)/2), \quad (t,q) \in [0,T] \times [-1,1].
$$
Note that we reversed time and introduced the variable $q=2p-1$ to achieve symmetry. Then, a direct calculation shows that $\mathcal{V}$ solves
\begin{equation}
\label{eq:conservation-law}
\begin{cases}
\partial_t \mathcal{V}(t,q) + \partial_{q} \mathfrak{g}(q, \mathcal{V}(t,q)) = 0, \\
\mathcal{V}(0,q) = \kappa q
\end{cases}
\end{equation}
for $(t,q) \in [0,T] \times [-1,1]$, where
$$
\mathfrak{g}(q, v) = q \frac{v |v|}{2} - \frac{v^2}{2}, \quad (q, v) \in [-1,1] \times \mathbb{R}.
$$
The paper \cite{cecchin2019convergence} shows that, for $\kappa=2$, the equation \eqref{eq:conservation-law} admits a unique entropy solution, and that, for large time horizons, a discontinuity emerges at the origin $q=0$.

\emph{Second MFG model.} Our second model features a running coupling cost but no terminal cost, and we follow~\cite{hofer2025synchronization} and~\cite{bayraktar2020non}, which provide an analysis of the corresponding mean-field NLL equation using the theory of scalar conservations laws. Here, we set
$$
\cX=\{0,1\},\quad \lambda^0\equiv \sigma^2,\quad \lambda^1\equiv1,\quad \ell(x,\mu,a)=\frac{1}{2} a^2 + \kappa\big(\mu(\{1\})\chi_{x=0}+\mu(\{0\})\chi_{x=1}\big),\quad g(x,\mu) \equiv 0,
$$
for $(x,\mu,a)\in\cX\times\Sigma^1\times\R_+$, where $\sigma^2\geq0$ is the \emph{thermal noise} and $\kappa>0$ is again the \emph{coupling strength}. The running cost $\ell(x,\mu,a)$ encourages players to be of the same type as others. Following the computation above, we note that this mean-field game is not Lasry-Lions monotone. Consistent with this observation, \cite{hofer2025synchronization} shows that for any $(\sigma^2,\kappa)$ that satisfy 
$$
\kappa>4\sigma^4,
$$ 
the game admits multiple MFG equilibria for sufficiently large time horizons when started in the uniform distribution. 
Again with the identification of a measure $\mu=(\mu(\{0\}),\mu(\{1\}))\in\Sigma^1$ with its value $p = \mu(\{1\})\in[0,1]$, the mean-field NLL equation reads, for $(t,x,p) \in [0,T] \times \cX \times [0,1]$, 
\begin{align*}
-\partial_t \mathcal{U}(t,x,p) &= \kappa(p\chi_{x=0}+(1-p)\chi_{x=1})\\
&\quad -
\frac{1}{2} \left[ \left( \cU(t,1-x,p)-\cU(t,x,p) \right)^- \right]^2 
+\sigma^2(\cU(t,1-x,p)-\cU(t,x,p))\\
&\quad -\partial_p \cU(t,x,p)\big([\cU(t,0,p)-\cU(t,1,p)]^-+\sigma^2\big)p\\
&\quad +\partial_p \cU(t,x,p)\big([\cU(t,0,p)-\cU(t,1,p)]^++\sigma^2\big)(1-p),
\end{align*}
and with terminal condition $\mathcal{U}(T,x,p) = 0$, see~\cite{bayraktar2020non}.
This system can be recognized as a scalar conservation law by introducing
$$
\mathcal{V}(t,q) = \mathcal{U}(T - t, 0,(q+1)/2) - \mathcal{U}(T - t, 1, (q+1)/2), \quad (t,q) \in [0,T] \times [-1,1].
$$
Note that we again reversed time and set $q=2p-1$. Then, a direct calculation shows that $\mathcal{Z}$ solves
\begin{equation}
\label{eq:conservation-law2} 
\begin{cases}
\partial_t \mathcal{V}(t,q) + \partial_{q} \mathfrak{g}(q, \mathcal{V}(t,q)) = 0, \\
\mathcal{V}(0,q) = 0
\end{cases}
\end{equation}
for $(t,q) \in [0,T] \times [-1,1]$, where
$$
 \mathfrak{g}(q, v) = 2\sigma^2qv+ q \frac{v |v|}{2} - \frac{v^2}{2} - \kappa\frac{q^2}{2}, \quad (q, v) \in [-1,1] \times \mathbb{R}.
$$
The work~\cite{bayraktar2020non} shows that for $\kappa=1$ and $\sigma^2<1/2$ and there is a unique entropy solution which may only jump at the uniform distribution, corresponding to $q=0$. This result can be extended to arbitrary values of $\kappa>0$.\\

We now compare different methods based on ODE solvers to approximate solutions to the mean-field NLL equation using the finite-player problem. In the following, we use the Python package \texttt{scipy.integrate.RK45} which is based on an explicit Runge-Kutta method of order 5(4).\\

\emph{Numerical results for the first model.} For the first model, we numerically show that the Picard algorithm based on an ODE solver coincides with the numerical solution of the \eqref{eq:N-NLL}. We choose the parameters 
$$
N=500,\quad T=1,\quad \Delta t = 0.01,\quad \kappa=2.
$$
Here, $\Delta t$ is the step size used to discretize the time interval $[0,T]$. In Figure \ref{fig:cecchin-model}, we plot the difference of the value functions $v^N(t,0)-v^N(t,1)$ at initial time $t=0$, which also coincides the optimal control of the tagged player, with the analytical solution in the mean-field limit. The plot in Figure \ref{fig:cecchin-model-NLL} is obtained by directly solving the \eqref{eq:N-NLL} equation, while the plot in Figure \ref{fig:cecchin-model-picard} uses the Picard iteration approach.
\begin{figure}[h]
    \centering
    \begin{subfigure}[t]{0.48\textwidth}
        \centering
        \includegraphics[width=\linewidth]{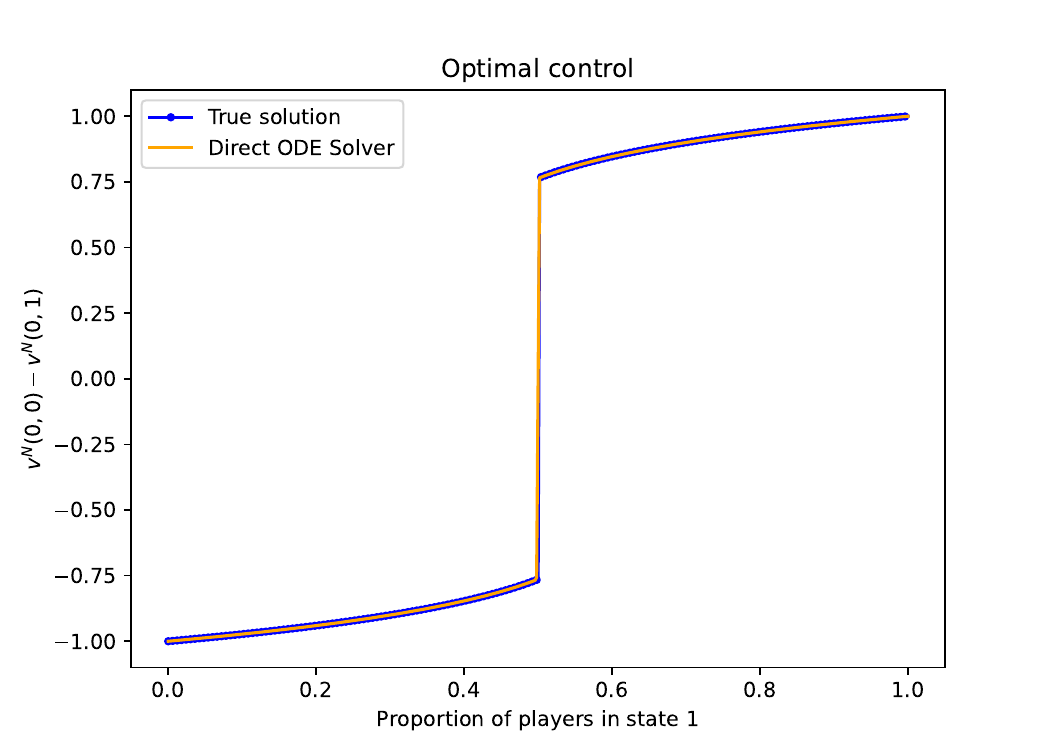}
        \caption{Direct solving~\eqref{eq:N-NLL}}
        \label{fig:cecchin-model-NLL}
    \end{subfigure}
    \hfill
    \begin{subfigure}[t]{0.48\textwidth}
        \centering
        \includegraphics[width=\linewidth]{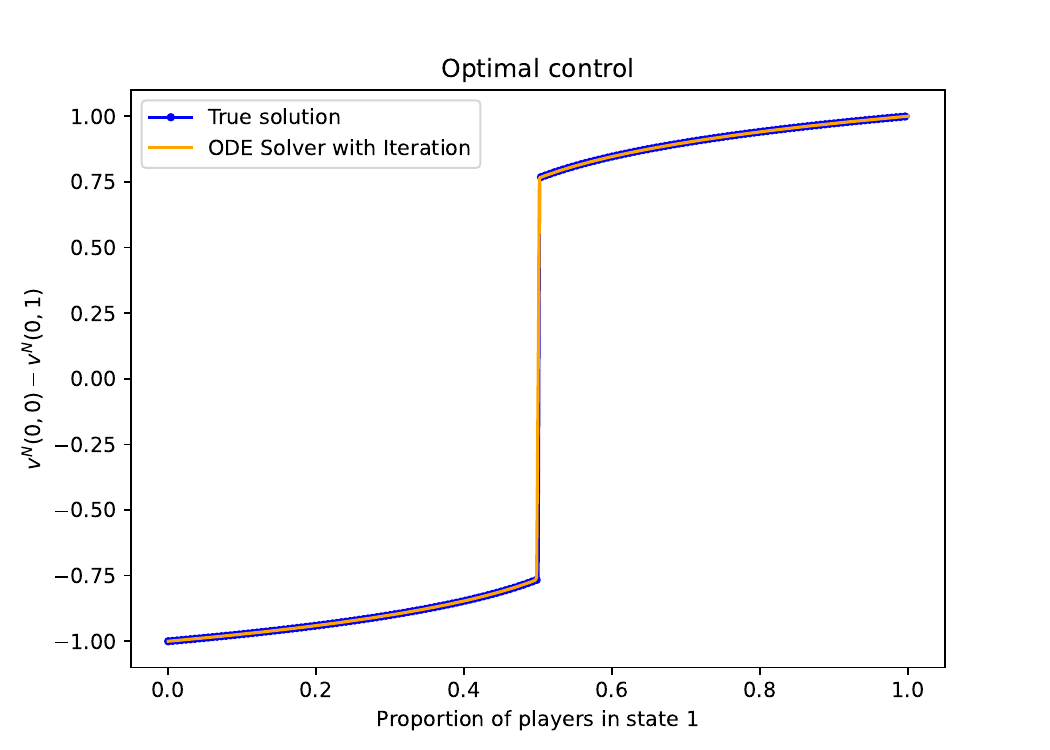}
        \caption{Picard iteration based on HJB equation}
        \label{fig:cecchin-model-picard}
    \end{subfigure}
    \caption{First model: plot of $p\mapsto v^N(t,0,p)-v^N(t,1,p)$ at initial time $t=0$.}
    \label{fig:cecchin-model}
\end{figure}

Next, we show numerically the convergence rates of Picard and weighted Picard iterations.  

\begin{figure}[h]
    \centering
    \begin{subfigure}[t]{0.48\textwidth}
        \centering
        \includegraphics[width=\linewidth]{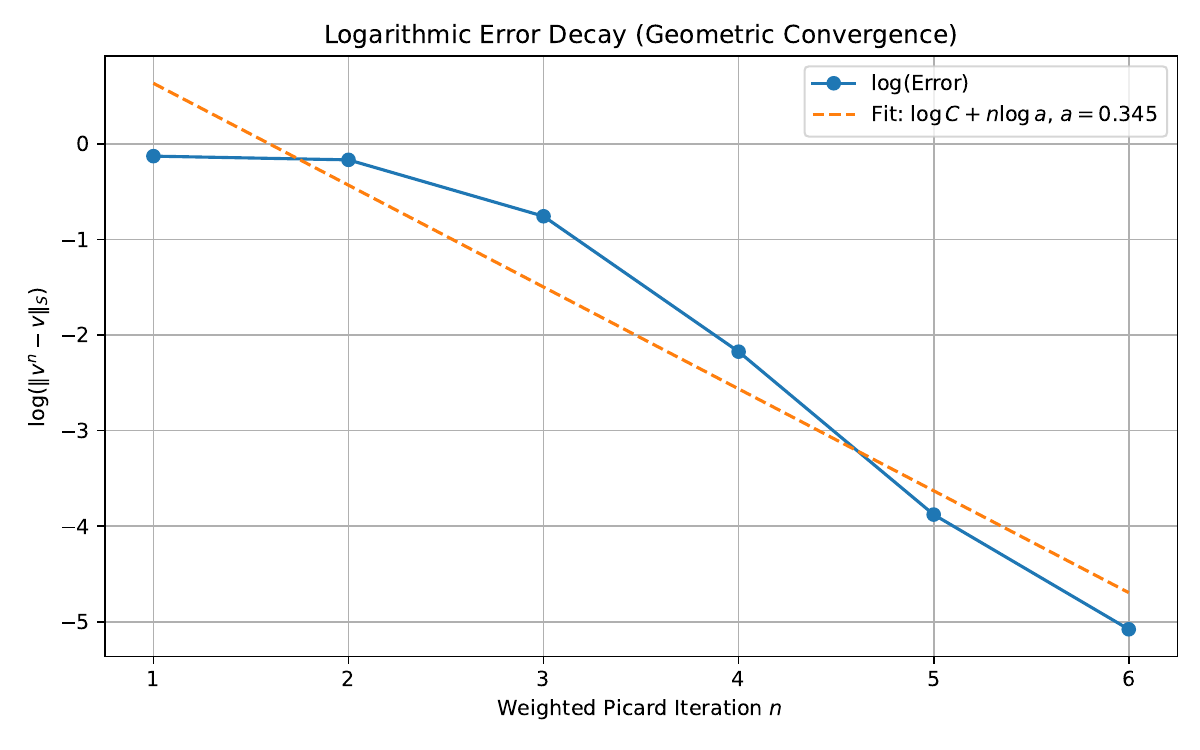}
        \caption{Picard iteration}
    \end{subfigure}
    \hfill
    \begin{subfigure}[t]{0.48\textwidth}
        \centering
        \includegraphics[width=\linewidth]{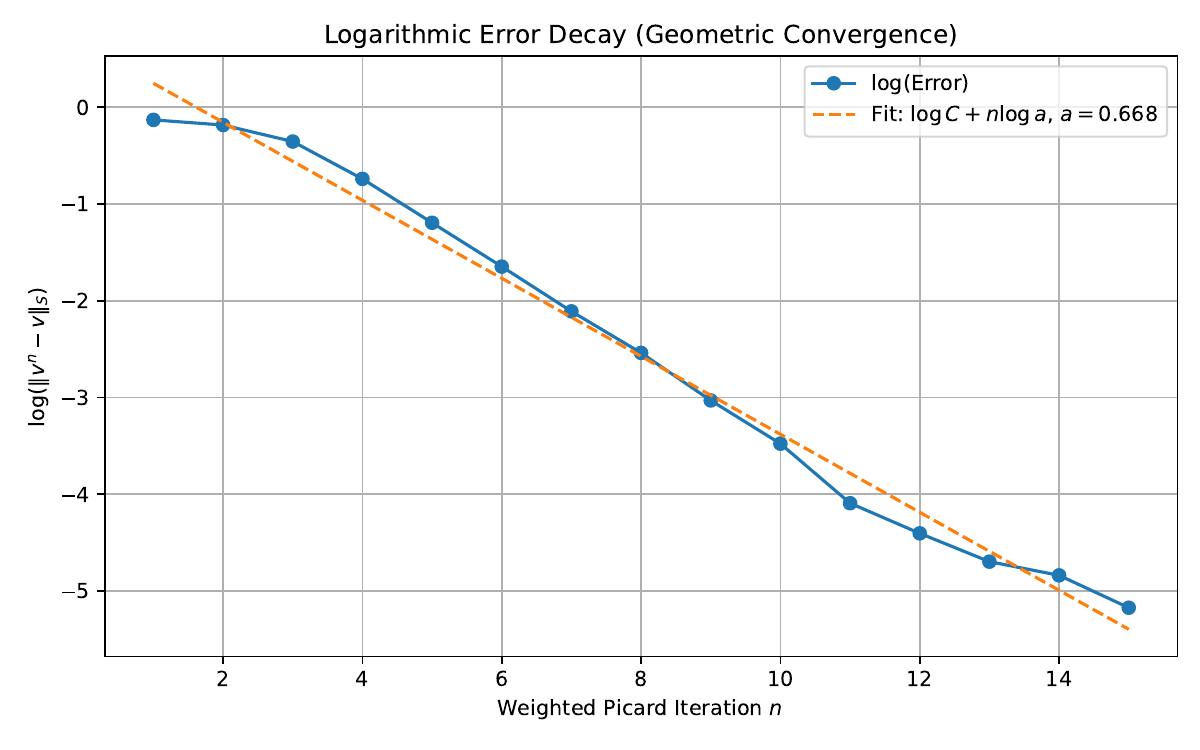}
        \caption{Weighted Picard iteration, $\we=0.5$}
    \end{subfigure}
    \caption{Convergence rates of the Picard and weighted Picard iterations ($N=100$)}
    \label{fig:convergence}
\end{figure}

For $N=100$, Figure~\ref{fig:convergence} shows the convergence rate of the value functions under Picard and weighted Picard iterations. 
We then use linear regression to fit it with a linear model corroborating the theoretical rates derived in Proposition \ref{prop:rate}.
Furthermore, the numerical results indicate that the convergence rate of the 
weighted Picard iteration is slower than the one of the Picard iteration.\\

%\newpage

\emph{Numerical results for the second model.} When solving the finite-player NLL equation \eqref{eq:N-NLL} in the second model directly for a large number of players $N$, we observed numerical instability for parameter regimes in which  the Picard algorithm still appears to be stable. We choose the following parameters to illustrate this phenomenon:
$$
N=500,\quad T=10,\quad \Delta t= 0.01,\quad  \kappa=6,\quad \sigma^2 = 0.5.
$$
Here, $\Delta t$ again denotes the step size used to discretize the time interval $[0,T]$. For these parameters, the left plot Figure~\ref{fig:two-state-NLL} is obtained by numerically solving~\eqref{eq:N-NLL} directly. We see that the jump occurs at the wrong location (not $1/2$), contradicting the theoretical results by \cite{bayraktar2020non}. The right plot Figure~\ref{fig:two-state-picard} shows the result after twelve Picard iterations, based on numerically solving the HJB equation, with the jump correctly occurring at $p=1/2$.

\begin{figure}[htbp]
    \centering
    \begin{subfigure}[t]{0.48\textwidth}
        \centering
        \includegraphics[width=\linewidth]{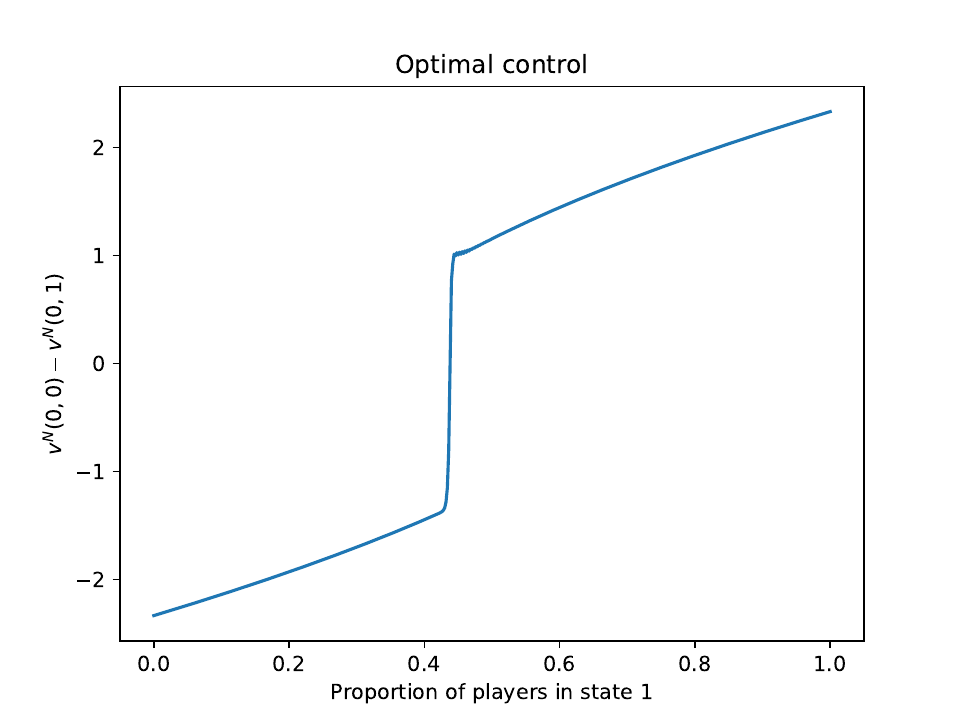}
        \caption{Directly solving~\eqref{eq:N-NLL}}
        \label{fig:two-state-NLL}
    \end{subfigure}
    \hfill
    \begin{subfigure}[t]{0.48\textwidth}
        \centering
        \includegraphics[width=\linewidth]{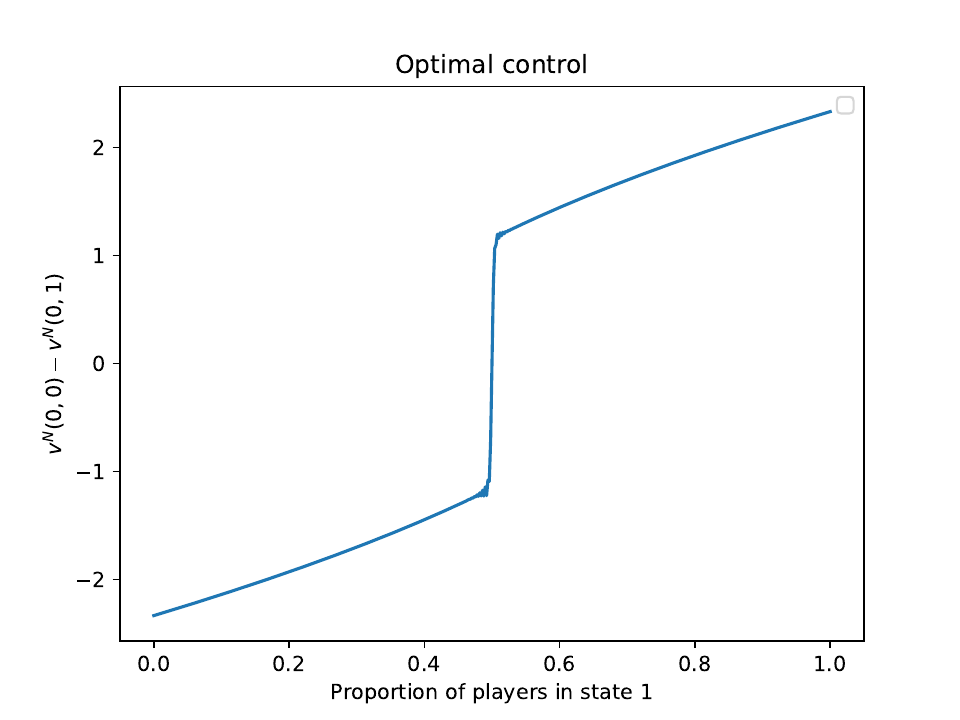}
        \caption{$12$ Picard iterations based on HJB equation}
        \label{fig:two-state-picard}
    \end{subfigure}
    \caption{Second model: plot of $p\mapsto v^N(t,0,p)-v^N(t,1,p)$ at initial time $t=0$.}
    \label{fig:empirical-distributions2}
\end{figure}

\subsection{Cyber-security model with adjustment costs}
\label{ss:cyber}

This section presents a variation of the cyber-security model due to Bensoussan and Kolokoltsov~\cite{Kolokoltsov_2016}. In this model, we are given a network of computers, each of which  can be in one of four states:
$$
\cX := \{\mathrm{DI}, \mathrm{DS}, \mathrm{UI}, \mathrm{US}\}.
$$
Here, the first letter refers to the protection status of the computer (defended versus undefended), and the second letter refers to whether the computer has been infected or is susceptible to infection. The model takes as given an attack rate $v_H>0$, infection rates $q_\mathrm{inf}^{\mathrm{D}} , q_\mathrm{inf}^{\mathrm{U}}>0$, recovery rates $q^{\mathrm{D}}_{\mathrm{rec}}, q^{\mathrm{D}}_{\mathrm{rec}}>0$ and parameters $\lambda_{\mathrm{UU}},\lambda_{\mathrm{DU}},\lambda_{\mathrm{DD}},\lambda_{\mathrm{UD}}>0$. A computer can get infected as a result of the attack rate $v_{\mathrm{H}}$, or from other infected computers. The model of Kolokoltsov and Bensoussan utilizes bang-bang optimal controls. Here, we modify slightly the model by considering an unbounded action set and a quadratic running cost which discourages from using large actions. 
The dynamics are determined by the following controlled transition rate matrix: Given a distribution $\mu=(\mu(\mathrm{DI}),\mu(\mathrm{DS}),\mu(\mathrm{UI}),\mu(\mathrm{US}))$ and a control $\alpha=\alpha(x,\mu)$ consider the following rate matrix
$$
\begin{array}{c|cccc}
 & \mathrm{DI} & \mathrm{DS} & \mathrm{UI} & \mathrm{US} \\
\hline
\mathrm{DI} & \cdots & q^{\mathrm{D}}_{\mathrm{rec}} & \alpha_{\mathrm{UI}}(\mathrm{DI},\mu) & 0 \\
\mathrm{DS} & v_{\mathrm{H}} q_\mathrm{inf}^{\mathrm{D}} + \lambda_{\mathrm{DD}} \mu(\mathrm{DI}) + \lambda_{\mathrm{UD}} \mu(\mathrm{UI}) & \cdots & 0 & \alpha_{\mathrm{US}}(\mathrm{DS},\mu) \\
\mathrm{UI} & \alpha_{\mathrm{DI}}(\mathrm{UI},\mu) & 0 & \cdots & q^{\mathrm{U}}_{\mathrm{rec}} \\
\mathrm{US} & 0 & \alpha_{\mathrm{DS}}(\mathrm{US},\mu) & v_{\mathrm{H}} q_\mathrm{inf}^{\mathrm{U}} + \lambda_{\mathrm{UU}} \mu(\mathrm{UI}) + \lambda_{\mathrm{DU}} \mu(\mathrm{DI}) & \cdots
\end{array}
$$
together with the following cost functional
$$
J(0,x,\mu,\alpha;\beta) = \E^{\alpha\otimes\beta}\Bigl[\int_0^T \left(\frac12 |\alpha(t,X_t,\mu_t)|_2^2 + (k_{\mathrm{D}}\chi_{\mathrm{D}}+k_{\mathrm{I}}\chi_{\mathrm{I}})(X_t) \right) \,\d t\Bigr],\quad (x,\mu)\in\sX,\, \alpha,\beta\in\sA.
$$
Here, $k_{\mathrm{D}}>0$ is the fee of having protection, $k_{\mathrm{I}}>0$ is the cost per unit of time resulting from infection and $\mathrm{D}:=\{\mathrm{DI},\mathrm{DS}\}$, $\mathrm{I}:=\{\mathrm{DI},\mathrm{UI}\}$.
The idea is that $\alpha$ determines the rate of changing the protection status of one's computer, from unprotected to protected and vice versa. This is costly and we model this using convex adjustment costs $\frac12\alpha^2$. The dots ``$\cdots$'' in the above transition matrix are the negative of the sum of the remaining row entries.

Since the number of ODEs is of order $(d-1)^N = 3^N$ in this example, instead of relying on an ODE solver, we resort to the deep learning approach outlined in Section~\ref{ssec:neural-network}: We use Monte Carlo simulations with empirical risk minimization for optimization in each iteration. The controls are implemented as fully connected feedforward neural networks with two hidden layers, each with 64 neurons.

The number of computers (players) is $24$, and the time horizon is $T=10$. We select the following parameters, which follow \cite[Section 7.2]{lauriere2021numerical},
\begin{align*}
v_{\mathrm{H}} = 0.2,\   q_{\mathrm{inf}}^{\mathrm{D}} = 0.4,\ q_{\mathrm{inf}}^{\mathrm{U}}=0.3,\   q_{\mathrm{rec}}^{\mathrm{D}}=0.1,\ q_{\mathrm{rec}}^{\mathrm{U}}=0.65,\ \\
\lambda_{\mathrm{UD}} = \lambda_{\mathrm{DD}}=0.4,\    \lambda_{\mathrm{DU}} = \lambda_{\mathrm{UU}}=0.3,\ k_{\mathrm{D}}=0.3,\  k_{\mathrm{I}}=0.5.   
\end{align*}

We start with two different initial distributions: (a) a uniform distribution over all states, see Figure \ref{fig:cyber-i}, and (b) with only defended and infected computers, see Figure \ref{fig:cyber-ii}. After training the neural network using Algorithm~\ref{algo:nn-picard}, we evaluate the code $10$ times and plot the empirical average for each of the four states with one standard deviation over 
these runs in Figure~\ref{fig:cyber}.

\begin{figure}[h]
    \centering
    \begin{subfigure}[t]{0.48\textwidth}
        \centering
        \includegraphics[width=\linewidth]{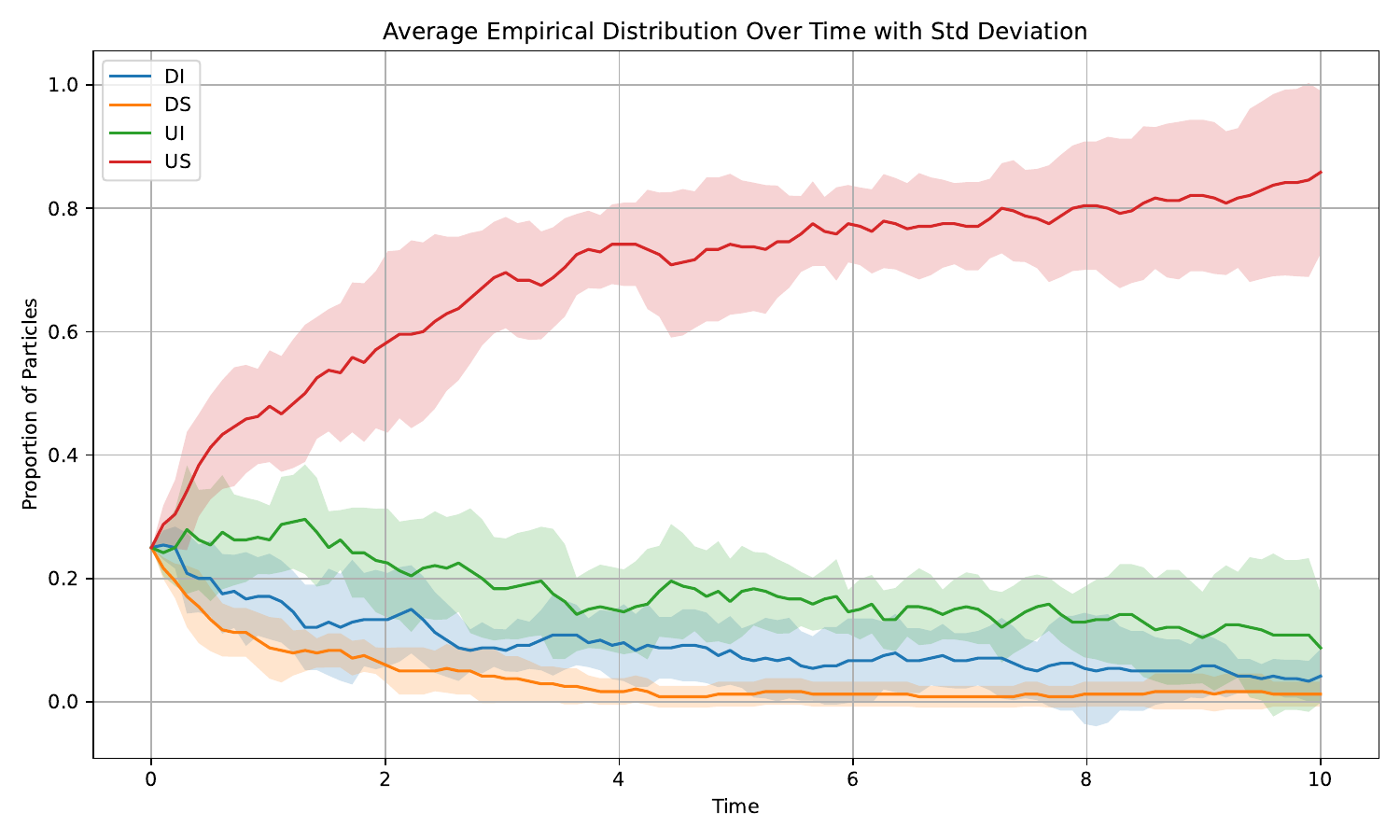}
        \caption{Starting from the uniform distribution}
        \label{fig:cyber-i}
    \end{subfigure}
    \hfill
    \begin{subfigure}[t]{0.48\textwidth}
        \centering
        \includegraphics[width=\linewidth]{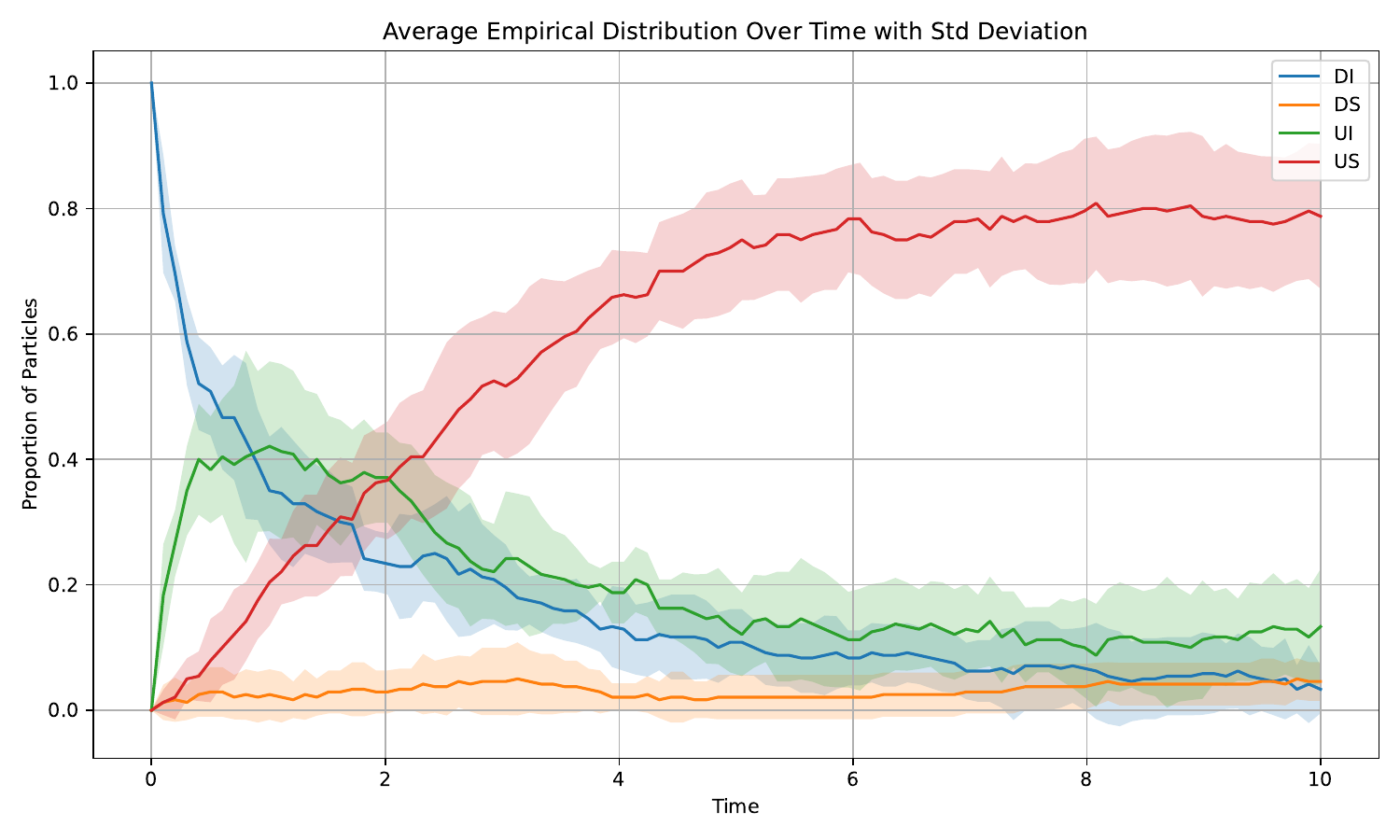}
        \caption{Starting from only defended and infected computers}
        \label{fig:cyber-ii}
    \end{subfigure}
    \caption{Time evolution of the empirical distributions}
    \label{fig:cyber}
\end{figure}

 We observe that our results are qualitatively consistent with the numerical results obtained in the mean-field model by \cite[Section 7.2]{lauriere2021numerical}, specifically, see their Figures 20 and 21. In particular, in both cases, the long time behavior is similar, with roughly $80\% $ of computers being not defended and susceptible to infection.

\bibliographystyle{abbrvnat}
\bibliography{ref}
\end{document}